\documentclass[a4paper]{amsart}
\usepackage[leqno]{amsmath}
\usepackage{amssymb}
\usepackage{amscd}
\usepackage{amsthm}
\usepackage{mathtools}
\usepackage{bbm}
\usepackage{color}
\usepackage{enumerate}
\usepackage{cite}
\usepackage[utf8]{inputenc}
\usepackage[all,cmtip]{xy}
\usepackage{tikz}
\usetikzlibrary{arrows}

\DeclareMathOperator{\GL}{GL}
\DeclareMathOperator{\PGL}{PGL}

\DeclareMathOperator{\Res}{Res}

\DeclareMathOperator{\Ind}{Ind}
\DeclareMathOperator{\Dist}{Dist}
\DeclareMathOperator{\St}{St}

\DeclareMathOperator{\cind}{c-ind}

\DeclareMathOperator{\tr}{tr}
\DeclareMathOperator{\ord}{ord}

\DeclareMathOperator{\sm}{sm}
\DeclareMathOperator{\sing}{sing}
\DeclareMathOperator{\coh}{coh}
\DeclareMathOperator{\cyc}{cyc}
\DeclareMathOperator{\al}{al}
\DeclareMathOperator{\sgn}{sgn}
\DeclareMathOperator{\vol}{vol}
\DeclareMathOperator{\ur}{ur}

\DeclareMathOperator{\fdhar}{fd,har}
\DeclareMathOperator{\ES}{ES}
\DeclareMathOperator{\ev}{ev}
\DeclareMathOperator{\HH}{H}
\DeclareMathOperator{\BS}{BS}
\DeclareMathOperator{\Hom}{Hom}
\DeclareMathOperator{\dc}{d}

\newcommand{\N}{{\mathbb N}}
\newcommand{\rec}{\operatorname{rec}}
\newcommand{\cf}{{\mathbbm 1}}
\newcommand{\Z}{{\mathbb Z}}
\newcommand{\Q}{{\mathbb Q}}
\newcommand{\C}{{\mathbb C}}
\newcommand{\R}{{\mathbb R}}
\newcommand{\F}{{\mathbb F}}

\newcommand{\A}{{\mathbb A}}
\newcommand{\I}{{\mathbb I}}
\renewcommand{\det}{\operatorname{det}}
\newcommand{\id}{\operatorname{id}}
\renewcommand{\Dist}{\operatorname{Dist}}
\newcommand{\iDist}{\operatorname{IDist}}
\newcommand{\SO}{\operatorname{SO}}
\renewcommand{\O}{\operatorname{O}}
\newcommand{\Sym}{\operatorname{Sym}}

\newcommand{\bound}{{\operatorname{b}}}

\newcommand{\p}{{\mathfrak{p}}}

\newcommand{\into}{\hookrightarrow}

\newcommand{\too}{\longrightarrow}								
\newcommand{\mapstoo}{\longmapsto}

\theoremstyle{plain}
\newtheorem{Lem}{Lemma}[section]
\makeatletter
\newlength{\@thlabel@width}%
\newcommand{\thmenumhspace}{\settowidth{\@thlabel@width}{\itshape1.}\sbox{\@labels}{\unhbox\@labels\hspace{\dimexpr-\leftmargin+\labelsep+\@thlabel@width-\itemindent}}}
\makeatother
\newtheorem{Pro}[Lem]{Proposition}
\newtheorem{Thm}[Lem]{Theorem}
\newtheorem{Def}[Lem]{Definition}
\newtheorem{Rem}[Lem]{Remark}

\newtheorem{Cor}[Lem]{Corollary}
\newtheorem{Exa}[Lem]{Example}

\newtheorem{Ass}{Assumption}

\title{On Shalika models and p-adic L-functions}

\author[L. Gehrmann]{Lennart Gehrmann}
\address{L. Gehrmann \\ Fakult\"at f\"ur Mathematik \\ Universit\"at Duisburg-Essen \\ Thea-Leymann-Stra\ss e 9 \\ 45127 Essen \\ Germany}
\email{lennart.gehrmann@uni-due.de}

\begin{document}

\begin{abstract} We use modular symbols to construct p-adic L-functions for cohomological cuspidal automorphic representations on GL(2n), which admit a Shalika model.
Our construction differs from former ones in that it systematically makes use of the representation theory of $p$-adic groups.
\end{abstract}
\maketitle
\tableofcontents

\section*{Introduction}
Modular symbols and integral formulas for special values of $L$-function were first used by Mazur-Swinnerton-Dyer \cite{MSD} and Manin \cite{Man} to construct $p$-adic $L$-functions for cuspidal elliptic eigenforms.
Besides generalizing the construction to automorphic representations on $\GL_{2}$ over other fields than the rationals (cf.~\cite{Man2}) there are several directions in which one can generalize the method to higher rank groups:

One can study the $p$-adic behaviour of Rankin-Selberg $L$-functions associated to cuspidal automorphic representations on $\GL_n\times \GL_{n-1}$ as has been done most notably by Schmidt \cite{Sch}, Kazhdan-Mazur-Schmidt \cite{KMS} and Januszewski \cite{Jan0}, \cite{Jan}.
In \cite{Mahnkopf} Mahnkopf studies the case of the standard $L$-function of a cuspidal representation of $\GL_3$.

In this paper we consider the standard $L$-function of a cuspidal automorphic representation $V$ on $\GL_{2n}$, which admits a Shalika model, over a totally real field $F$.
The construction of a $p$-adic $L$-function in this setting has been carried out by Ash-Ginzburg \cite{AG} under the following assumptions:
\begin{enumerate}[(A)]
  \item The cuspidal automorphic representation is cohomological with respect to the trivial representation. \label{a}
  \item The local components $V_{\p}$ are spherical and ordinary for all places $\p$ above $p$. \label{b}
	\item A certain restriction on the $p$-class group of $F$. \label{c}
\end{enumerate}
The main aim of this article is to relax the above assumptions.
In particular, we get rid of assumption \eqref{c} completely and work with a cuspidal representation, which is cohomological with respect to an arbitrary algebraic representation $V_{\al}$.
Our methods are somewhat different from the ones of Ash and Ginzburg. Whereas their arguments are close to the classical ones, i.e.~using Hecke relations to prove the distribution property and the boundedness of the distribution, we follow the strategy of Spie\ss~in \cite{Sp}, which builds on work of Darmon (cf.~\cite{D}).
The two main features of this construction are the following: All computations, e.g.~the distribution and interpolation property, are purely local and we construct the distribution from a cohomology class in the group cohomology of an $S_{p}$-arithmetic subgroup of $\GL_{2n}$ with values in $(V_{p}\otimes V_{\al})^{\vee}$, where $V_{p}$ is the tensor product $\otimes_{\p\in S_{p}}V_{\p}$.
The existence of a well-behaved lattice inside the locally algebraic representation $V_{p}\otimes V_{\al}$ then implies the boundedness of the distribution.

The main results of the paper are:
\begin{itemize}
\item Definition of the distribution in Section \ref{globaldist}
\item Proof of an interpolation property (see Proposition \ref{interpo})
\item Proof of rationality of the distribution (see Corollary \ref{distrational})
\item Proof of integrality of the distribution (see Corollary \ref{distintegral})
\end{itemize}
Let us explain our results and their proofs in more detail:

The heart of the article is Section \ref{heart}.
We develop a theory of stabilizations with respect to the standard parabolic subgroup $P_{n}\subset \GL_{2n}$ of type $(n,n)$ in terms of the representation theory of $p$-adic groups.
More precisely, a local component $V_{\p}$ of $V$ admits a stabilization $\Theta$ if it is a quotient of a parabolically induced representation from $P_n$.
Thus, we can view $V_{\p}$ as the global sections of a sheaf on the quotient $P_{n}(F_{\p})\backslash\GL_{2n}(F_{\p})$.
We use the action of the Levi subgroup of $P_n$ to trivialize this sheaf on an open subset inside the open Bruhat cell.
This yields a map
$$\delta_{\Theta}\colon C_{c}^{0}(\GL_{n}(F_{\p}),\C)\too V_{\p}.$$
In the main Lemma \ref{integrality 2} we show that $\delta_{\Theta}$ (resp.~a variant of it for locally algebraic representations) respects integral structures on both sides as long as the stabilization $\Theta$ is \textit{weakly ordinary}.
This weak ordinarity condition is equivalent to the usual ordinarity condition in the $\GL_{2}$-case but strictly weaker in the higher rank case.
In Section \ref{localshalika} we define a modified Euler factor $E(\Theta,\chi_\p,s)$ for every character $\chi_\p\colon F_{\p}^{\ast}\to \C^{\ast}$ by integrating $\chi\circ\det$ over the pullback of the local Shalika functional by $\delta_{\Theta}$.
The modified Euler factors are holomorphic multiples of the local $L$-factors, i.e.~we have
$$
E(\Theta,\chi_\p,s)=e(\Theta,\chi_\p,s)\cdot L(V_\p\otimes\chi_\p,s),
 $$ where $e(\Theta,\chi_\p,s)$ is an entire function. Using explicit formulas for the Shalika functional we show that the modified Euler factors are the expected ones if $V_{\p}$ is a twist of an unramified principal series representation by a character. 

In the third section we globalize our construction.
Let $S_p$ be the set of places of $F$ above $p$.
Suppose we have given stabilizations $\Theta_{\p}$ of $V_{\p}$ for all places $\p\in S_{p}$, a critical half integer $s+1/2$ and a finite set of finite places $\Sigma$, which is disjoint from $S_{p}$. Then we construct a distribution $\mu_{\Theta_{\Sigma},s}$ on the Galois group $\mathcal{G}_{p}$ of the maximal abelian extension of $F$ unramified outside $p$ and $\infty$ such that the interpolation property
\begin{align*}
\int_{\mathcal{G}_{p}}{\chi(\gamma)\mu_{\Theta_{\Sigma},s}(d\gamma)}=\prod_{\p \in S_{p}}e(\Theta_{p},\chi_{p},s+1/2) \times L_{\Sigma}(V\otimes\chi,s+1/2)
\end{align*}
holds for all continuous characters $\chi\colon \mathcal{G}_p\to \C^{\ast}$.

Finally, in the last section we recast the definition of $\mu_{\Theta_{\Sigma},s}$ in terms of modular symbols. As an immediate consequence we get that the distribution takes values in a finite dimensional vector space over the field of definition $\mathcal{E}$ of the finite part of $V$.
Using results of Grobner and Raghuram on rationality of Shalika models (cf.~\cite{GRG}) we show the following more refined result: there exist periods $\Omega^{\varepsilon}$ for all characters $\varepsilon\colon F^{\ast}_\infty=(F\otimes\R)^{\ast}\to\left\{\pm 1\right\}$ such that
$$\int_{\mathcal{G}_{p}}{\chi(\gamma)\mu_{\Theta_{\Sigma},s}(d\gamma)} \in \mathcal{E}_{\chi}^{\prime}\Omega^{\chi_{\infty}}.$$
Here $\mathcal{E}^{\prime}\subset \C$ is the smallest extension of $\mathcal{E}$, over which the stabilizations $\Theta_{\p}$ are defined, and $\mathcal{E}^{\prime}_{\chi}$ is the field you get by adjoining the image of $\chi$.
Assuming that the stabilizations are weakly ordinary we show that the distributions $\mu_{\Theta_{\Sigma},s}$ are $p$-adically bounded provided that the locally algebraic representation $V_{\al}\otimes V_{p}$ admits a lattice $L$, which has a resolution by compactly induced representations of finite type (see Definition \ref{re} for a precise definition).
The two main steps in proving the boundedness are:
\begin{itemize}
	\item Arithmetic groups are of type (VFL) (as introduced by Serre in \cite{Se2}) and therefore, modular symbols with values in the lattice $L$ commute with flat base change.
	\item $\delta_{\Theta}$ respects integral structures.
\end{itemize}
We end the introduction by some comments on the existence of the above mentioned lattices: By the Breuil-Schneider conjecture the representation $V_{\al}\otimes V_{p}$ should always admit some lattice.
In the $\GL_{2}$-case Vignéras has shown in \cite{Vi} that the existence of some lattice is equivalent to the existence of a lattice, which has a good resolution.
In Section \ref{principal} we list known examples of locally algebraic representations in the higher rank case, which admit good lattices. Since smooth ordinary principal series representations admit good lattices by the work of Ollivier (cf.~\cite{Ol}) our construction covers all cases discussed in \cite{AG}.
In the higher weight case the existence of good lattices is known in special cases by results of Gro\ss e-Kl\"onne (see \cite{GK}).

\bigskip
\textbf{Acknowledgments}.
I am grateful to Michael Spie\ss~for sharing his ideas on the construction of $p$-adic $L$-functions.
I thank Jan Kohlhaase, Andreas Nickel and Vytautas Pa\v{s}k\=unas for several helpful discussions and Felix Bergunde for a thorough reading of an earlier draft of the article.
It is a pleasure to thank all past and current members of the Bielefeld arithmetic geometry study group for the daily all-important coffee break.
Finally, I like to thank the anonymous referee for a detailed list of remarks, which helped to improve the exposition of the article.

\bigskip
\textbf{Notions and Notations}. We will use the following notions and notations throughout the whole article.
At the beginning of each chapter there will be an additional set of notations which may be only valid for that given section.

The entry in the $i$-th row and $j$-th column of a matrix $A$ is denoted by $A_{ij}$. We denote the $n\times n$-identity matrix by $1_n$.

All rings are commutative and have a unit. The group of invertible elements of a ring $R$ will be denoted by $R^{\ast}$.
If $M$ is an $R$-module, we denote the dual module $\Hom_{R}(M,R)$ by $M^{\vee}$.

If $R$ is a ring and $G$ a group, we will denote the group ring of $G$ over $R$ by $R[G]$.
If $G$ is a topological group, we write $G^{\circ}$ for the connected component of the identity. Given a closed subgroup $H$ of a locally profinite group $G$ and an $R$-linear representation $M$ of $H$, the \textit{(smooth) induction} $\Ind^{G}_{H}M$ of $M$ from $H$ to $G$ is the space of all locally constant functions $f\colon G\to M$ such that $f(hg)=hf(g)$ for all $h\in H, g\in G$.
The induction $\Ind^{G}_{H}M$ is an $R$-module on which $G$ acts $R$-linearly via the right regular representation.
The \textit{(smooth) compact induction} $\cind^{G}_{H}M$ is the $R[G]$-submodule of $\Ind^{G}_{H}M$ consisting of functions which have compact support modulo $H$.
Let $\chi\colon G\to R^{\ast}$ be a character.
We write $R[\chi]$ for the $G$-representation, which underlying $R$-module is $R$ itself and on which $G$ acts via the character $\chi$.
Given a character $\chi\colon H\to R^{\ast}$ we will often write $\Ind^{G}_{H}\chi$ (resp.~$\cind^{G}_{H}\chi$) instead of $\Ind^{G}_{H}R[\chi]$ (resp.~$\cind^{G}_{H}R[\chi]$).

For a set $X$ and a subset $A\subset X$ the characteristic function $\cf_{A}\colon X\to \left\{0,1\right\}$ is defined by
\begin{align*}
\cf_{A}(x)=\begin{cases}
1 & \mbox{if } x\in A,\\
0 & \mbox{else}.
\end{cases}
\end{align*}
We fix a prime $p$ and embeddings
\begin{align*}
\C \xhookleftarrow{\iota_{\infty}} \overline{\Q} \xhookrightarrow{\iota_{p}} \C_{p}.
\end{align*}
We let $\ord_{p}$ denote the valuation on $\C_{p}$ (and on $\overline{\Q}$ via $\iota_{p}$) normalized such that $\ord_{p}(p)=1$.
The valuation ring of $\overline{\Q}$ with respect to $\ord_{p}$ will be denoted by $\overline{\mathcal{R}}$.

\begin{section}{Preliminaries on function and distribution spaces}
The purpose of this section is twofold.
Firstly, we want to fix notations.
Secondly, we want to collect results from the literature which we are going to use later. 
\begin{subsection}{Distributions and measures}\label{distmeas}
Given two topological spaces $X,Y$ we will write $C(X,Y)$ for the space of continuous functions from $X$ to $Y$.
If $R$ is a topological ring, we define $C_{c}(X,R) \subseteq C(X,R)$ as the subspace of continuous functions with compact support.
If we consider $Y$ (resp.~$R$) with the discrete topology, we will often write $C^{0}(X,Y)$ (resp.~$C_{c}^{0}(X,R)$) instead.

Since a locally constant map with compact support takes only finitely many different values, the canonical map
\begin{align}\label{compactfunc}
C^{0}_{c}(X,\Z)\otimes R\too C^{0}_{c}(X,R)
\end{align}
is an isomorphism of $R$-modules.
For a ring $R$ and an $R$-module $N$ the $R$-module of $N$-\textit{valued} \textit{distributions} on $X$ is given by $\Dist(X,N)=\Hom_{\Z}(C^{0}_{c}(X,\Z),N)$.
By \eqref{compactfunc} every distribution $\mu\in\Dist(X,N)$ extends uniquely to an $R$-linear homomorphism
\begin{align*}
C_{c}^{0}(X,R)\too N,\ f\mapstoo\int_{X}{f\ d\mu}.
\end{align*}
Suppose $f\colon X\to Y$ is a continuous map between compact spaces.
Then the pullback map
\begin{align*}
C^{0}(Y,R)\too C^{0}(X,R),\quad g\mapstoo g\circ f  
\end{align*}
induces a push forward map on distributions
\begin{align}\label{properpush}
f_{\ast}\colon\Dist(X,N)\too \Dist(Y,N).
\end{align}
If $G$ is a topological group and $H$ a closed subgroup, $G$ acts on $C^{0}(G/H,N)$ (resp.~$C^{0}_{c}(G/H,R))$ via left-multiplication, i.e.~$(gf)(x)=f(g^{-1}x)$.
This in turn induces a $G$-action on $\Dist(G/H,N)$ via $(gD)(f)=D(g^{-1}f)$.
In case $H$ is trivial, we can extend the action of $G$ to a $G\times G$-action on $C^0(G,N)$ (resp.~$C^{0}_{c}(G,R)$) by $(g_{1},g_{2})f(g)=f(g_{1}^{-1}gg_{2})$.

Now, let $\mathcal{X}$ be a profinite topological space and $E$ a $p$-adic field, i.e.~$E$ is a field of characteristic $0$ which is complete with respect to an absolute value $\left|\cdot\right|\colon E\to\R$ whose restriction to $\Q$ is the usual $p$-adic absolute value.
Until the end of this subsection $R$ will denote the valuation ring of $E$. 

Let $V$ be a finite dimensional $E$-vector space and $L\subset V$ a lattice.
The space $\Dist^{\bound}(\mathcal{X},V)$ of \textit{bounded distributions} is defined as the image of the inclusion $$\Dist(\mathcal{X},L)\otimes E\too \Dist(\mathcal{X},V).$$
The definition does not depend on the choice of lattice.
Any bounded $V$-valued distribution $\mu$ can uniquely be extended to a continuous $E$-linear homomorphism
$$
C(\mathcal{X},E)\too V,\ f\mapstoo \int_{\mathcal{X}}f\ d\mu.
$$
We say that a $\C$-valued distribution $\mu\in \Dist(\mathcal{X},\C)$ is a $p$-\textit{adic measure} if there exists a Dedekind ring $R\subset\overline{\mathcal{R}}$ such that the image of $C^{0}(\mathcal{X},\Z)$ under $\mu$ is contained in a finitely generated $R$-submodule of $\C$.
Let $L_{\mu,R}$ the smallest such $R$-submodule of $\C$.
In this case $\mu$ defines a bounded distribution with values in $\widetilde{L_{\mu}}:=L_{\mu,R}\otimes_{R}\C_{p}$.
\end{subsection}
\begin{subsection}{Smooth representations}
Let $G$ be a locally profinite group and $E$ a field of characteristic $0$.
A $G$-representation on an $E$-vector space $V$ is \textit{smooth} if the stabilizer of $v$ in $G$ is open for all $v\in V$
We write $\mathfrak{C}^{\sm}_{E}(G)$ for the category of smooth $G$-representations on $E$-vector spaces.

\begin{Lem}\label{projectivity}
Let $K\subset G$ be a compact, open subgroup. Then $\cind_{K}^{G}P$ is a projective object in $\mathfrak{C}^{\sm}_{E}(G)$ for every $K$-representation $P\in \mathfrak{C}^{\sm}_{E}(K)$.
\end{Lem}
\begin{proof}
As a consequence of Frobenius reciprocity the compact induction functor sends projective objects to projective objects. But $\mathfrak{C}^{\sm}_{E}(K)$ is a semi-simple abelian category (see for example Chapter 2.2 of \cite{BH}) and thus every object in $\mathfrak{C}^{\sm}_{E}(K)$ is projective.
\end{proof}

\end{subsection}

\begin{subsection}{Lattices and integral resolutions}\label{principal}
Let $E$ be a field, $R\subset E$ a subring and $G$ a locally profinite group.
A $G$-representation on an $E$-vector space $V$ is called $R$-\textit{integral} if there exists a locally-free $G$-stable $R$-submodule $L\subset V$ such that $V=L\otimes_R E$.
In this case we call $L$ an $R[G]$-\textit{lattice} in $V$.
For example, $C^{0}_c(G,R)$ is a $R[G\times G]$-lattice in $C^{0}_c(G,E)$.
There is also a twisted variant of this:
Suppose we have a character $\chi\colon G\to E^{\ast}$ and a compact, open subgroup $K\subset G$ such that $\chi(K)\subset R^{\ast}$.
Then, the $R$-module
\begin{align}\label{characterspace}
C_c^{0}(G,R)\otimes_{R}\chi := \cind^{G\times K}_{K\times K}(C^{0}(K,R)\otimes_R R[\chi])
\end{align}
defines a $R[G\times K]$-lattice in $C_c^{0}(G,E)\otimes_E E[\chi]$.

The second kind of examples we are interested in will be lattices in locally algebraic representations of reductive groups over $p$-adic fields:
Let $F/\Q_p$ be a finite extension with ring of integers $\mathcal{O}_F$ and $n\geq 1$ a fixed integer.
We write $G_n$ for the group of $F$-rational points of $\mathbb{GL}_{n}$ with the p-adic topology and $Z\subset G_{n}$ for its center.
Further, we fix a finite extension $E\subset \C_p$ of $\Q_p$ such that every embedding $F\hookrightarrow \C_p$ factors through $E$
 We will denote the valuation ring of $E$ by $R$.
\begin{Def}\label{re}
Let $V$ be a representation of $G_n$ on an $E$-vector space.
An $R[G_n]$-lattice $L$ inside $V$ is called \textit{homologically of finite type} if there exists a resolution of finite length
\begin{align}\label{res}
0\too C_m\too\ldots\too C_0\too L \too 0
\end{align}
of $R[G_n]$-modules with the following properties: Each $C_i$ is of the form
\begin{align*}
C_i=\cind_{ZK_{\left[i\right]}}^{G}L_i
\end{align*}
with compact, open subgroups $K_{\left[i\right]}\subset G_n$ and $R[ZK_{\left[i\right]}]$-modules $L_i$, which are free $R$-modules of finite rank. 
We say that $V$ is \textit{homologically integral} if $V$ admits a lattice which is homologically of finite type.
\end{Def}
\begin{Rem}
\begin{enumerate}[(i)]\thmenumhspace
    \item The significance of the notion of homological integrality will be made apparent in the proof of Proposition \ref{flat}.
		\item It is easy to see that the property of being homologically integral is preserved under twisting by finite order characters.
\end{enumerate}
\end{Rem}
\noindent An irreducible locally $\Q_p$-rational representation of $G_n$ on an $E$-vector space is a tensor product $V=V_{\sm}\otimes V_{\al}$, where $V_{\sm}\in\mathfrak{C}^{\sm}_{E}(G_n)$ is irreducible and $V_{\al}$ is an irreducible $E$-rational representation of the algebraic group $\Res_{F/\Q_p}(\mathbb{GL}_{n,F})$.
The following proposition shows that being homologically integral is rather common.
\begin{Pro}[Vignéras]\label{latticesv}
Let $V$ be an irreducible, locally $\Q_p$-rational representation of $G_2$ on an $E$-vector space.
Then $V$ is integral if and only if $V$ is homologically integral.
\end{Pro}
\begin{proof}
The locally algebraic case is the content of Proposition 0.4 of \cite{Vi}.
The locally $\Q_p$-rational case is proved in exactly the same way.
\end{proof}
\begin{Def}
Suppose that $F=\Q_p$.
An irreducible algebraic representation $V_{\al}$ is said to have $p$-small weights if it fulfills the following two conditions:
\begin{itemize}
  \item The reduction $\bmod\ p$ of one (and thus every) $R[\GL_n(\Z_p)]$-lattice of $V_{\al}$ is absolutely irreducible.
  \item We have $\left\langle \mu+\rho,\check{\beta}\right\rangle \leq p$ for every positive root $\beta$.
Here $\mu$ denotes the highest weight of $V_{\al}$ with respect to the Borel subgroup of upper triangular matrices and $\rho$ is half of the sum of all positive roots.
Equivalently, if we write $\mu=(\mu_1,\ldots,\mu_n)$, the above condition translates into
$$\mu_i-\mu_{i-1}+1\leq p$$
for all $1\leq i\leq n-1$.
\end{itemize}
\end{Def}
\begin{Thm}[Gro\ss e-Kl\"onne]\label{latticesg}
Let $V$ be an irreducible, locally $\Q_p$-rational representation of $G_n$ on an $E$-vector space and assume further that
\begin{itemize}
\item $F=\Q_p$,
\item $V_{\al}$ has $p$-small weights,
\item $V_{\sm}$ is an irreducible, unramified principal series representation,
\item the central character of $V$ takes values in $\Z_{p}^{\ast}$ and
\item the (twisted) Hecke-eigenvalues of $V$ are integral.
\end{itemize}
Then $V$ is homologically integral.
\end{Thm}
\begin{proof}
By the results in Section 3 of \cite{SK} our situation is just a special case of \cite{GK}, Theorem 1.1 (iii).
\end{proof}

\begin{Def}\label{orddef}
A principal series representation $V\in\mathfrak{C}^{\sm}_{E}(G_n)$ is called \textit{ordinary} if there exists a character
$\chi\colon B_n\to R^{\ast}$
on a Borel group $B_n\subset G_n$ such that
\begin{align*}
V\cong \Ind_{B_n}^{G_n}E[\chi].
\end{align*}
\end{Def}
\begin{Thm}[Ollivier]\label{latticeso}
Suppose $V\in\mathfrak{C}^{\sm}_{E}(G_n)$ is an ordinary irreducible principal series representation. Then $V$ is homologically integral.
\end{Thm}
\begin{proof}
Given a character $\chi\colon B_{n}\to R^{\ast}$ as in the definition of ordinarity we can consider the lattice
$$\Ind_{B_n}^{G_n}R[\chi]\subset V.$$
By the work of Ollivier (cf.~\cite{Ol}) this lattice is homologically of finite type.
Note that Ollivier only considers representations over fields in \textit{loc.~cit.} but her methods carry over verbatim to our situation.
\end{proof}

\end{subsection}
\end{section}

\begin{section}{Local distributions}\label{heart}
Throughout this section let $F$ be a finite extension of $\Q_{p}$, $\mathcal{O}$ its ring of integers with maximal ideal $\mathfrak{p}=(\varpi)$ and $q$ the cardinality of its residue field.
We denote the group of units of $\mathcal{O}$ by $U$ and put $U^{(m)}=\left\{u\in U|\ u\equiv 1 \bmod \mathfrak{p}^{m} \right\}$.
We denote by $\nu$ the normalized additive valuation on $F$ (i.e.~$\nu(\varpi)$=1) and by $\left|x\right|$ the modulus of $x\in F^{\ast}$ (i.e. $|\varpi|=q^{-1}$).

Let $G_{r}$ (resp.~$K_{r}$) denote the group of invertible $(r\times r)$-matrices over $F$ (resp.~$\mathcal{O}$) and let $d^{\ast}g$ denote the Haar measure on $G_{r}$ (normalized such that $K_r$ has volume $1$). 
We denote by $K_{r}^{(m)}$ the principal congruence subgroup of $K_{r}$ of level $m$, i.e.~the kernel of the reduction map from $K_{r}$ to $\GL_{r}(\mathcal{O}/\mathfrak{p}^{m})$.
The Borel subgroup of upper triangular matrices in $G_{r}$ will be denoted by $B_{r}$.
If $r=2n$ is even, we let $P_{n}$ denote the standard parabolic subgroup of $G_{2n}$ of type $(n,n)$.
Finally, we fix a character $\psi\colon F\to\overline{\Q}^{\ast}$ of conductor $\mathcal{O}$.
\begin{subsection}{The map $\delta$}\label{delta}
In this section we construct a map $\delta$ (depending on several choices) from the space of locally constant functions on $G_{n}$ to smooth representations of $G_{2n}$ which are parabolically induced from $P_{n}$. 
It will be used in Section \ref{localshalika} to define the local part of our global distribution. 
The map $\delta$ was first studied by Spie\ss~ for $n=1$ in \cite{Sp}.

Let us fix an irreducible representations $\pi\in\mathfrak{C}^{\sm}_E(G_n\times G_n)$ over a field $E$ of characteristic $0$.
We can write $\pi$ as a tensor product $\pi_{1}\otimes_E \pi_{2}$, where $\pi_{1},\ \pi_{2}\in \mathfrak{C}^{\sm}_E(G_n)$ are irreducible representations.
They are uniquely determined up to isomorphism by $\pi$. We consider $\pi$ as a representation of $P_{n}$ via the projection $P_{n}\twoheadrightarrow G_{n} \times G_{n}$.

For every element $\rho\in\pi$ we define the $E$-linear map 
\begin{align}\label{deltadef}
\delta=\delta_{\rho}\colon C^{0}_{c}(G_{n},E)\too \Ind_{P_{n}}^{G_{2n}}\pi
\end{align}
as follows: if $g\in G_{2n}$ is of the form
\begin{align*}
g=\begin{pmatrix}g_{1}&\ast\\0&g_{2}\end{pmatrix} \begin{pmatrix}0&1_n\\ 1_n&0\end{pmatrix} \begin{pmatrix}1_n&u\\0&1_n\end{pmatrix}
\end{align*}
with $g_{1},g_2, u\in G_{n}$, we put
$$\delta(f)(g)=f(u)\cdot \pi\left(g_{1},g_{2}u\right)\rho$$
and otherwise we set $\delta(f)(g)=0$.

The group $G_{n}\times G_{n}$ acts on $C^{0}_{c}(G_{n},E)$ as in Section \ref{distmeas} and on $\Ind_{P_{n}}^{G_{2n}}$ through the diagonal embedding of $G_{n}\times G_{n}$ into $G_{2n}$.
\begin{Lem}\label{equi}     
Let $\mathcal{D}\subset G_n$ be the subgroup given by $\mathcal{D}=\left\{g\in G_n\mid (g,g)\rho=\rho\right\}.$  Then $\delta$ is $G_{n}\times \mathcal{D}$-equivariant.
\end{Lem}
\begin{proof}
Let $g=\begin{pmatrix}g_{1}&\ast\\ 0&g_{2}\end{pmatrix} \begin{pmatrix}0&1_n\\ 1_n&0\end{pmatrix} \begin{pmatrix}1_n&u\\ 0&1_n\end{pmatrix}$ $\in G_{2n}$ with $g_{1},g_{2},u\in G_{n}$ and $(h_{1},h_{2})\in G_{n}\times \mathcal{D}$. For every $f\in C^{0}_{c}(G_{n},E)$ we have
\begin{align*}
\left(\begin{pmatrix} h_{1}&0\\0&h_{2}\end{pmatrix}\delta(f)\right)(g)
&=\delta(f)\left(\begin{pmatrix}g_{1}&\ast\\ 0&g_{2}\end{pmatrix} \begin{pmatrix}0&1_n\\ 1_n&0\end{pmatrix} \begin{pmatrix}1_n&u\\ 0&1_n\end{pmatrix}\begin{pmatrix} h_{1}&0\\0&h_{2}\end{pmatrix}\right)\\
&=\delta(f)\left(\begin{pmatrix}g_{1}h_{2}&\ast\\ 0&g_{2}h_{1}\end{pmatrix} \begin{pmatrix}0&1_n\\ 1_n&0\end{pmatrix}\begin{pmatrix}1_n&h_{1}^{-1}uh_{2}\\ 0&1_n\end{pmatrix}\right)\\
&=f(h_{1}^{-1}uh_{2})\cdot \pi(g_{1}h_{2},g_{2}h_{1}h_{1}^{-1}uh_{2})\rho\\
&=f(h_{1}^{-1}uh_{2})\cdot \pi(g_{1},g_{2}u)\pi(h_{2},h_{2})\rho\\
&=f(h_{1}^{-1}uh_{2})\cdot \pi(g_{1},g_{2}u)\rho\\
&=\delta((h_{1},h_{2})f)(g)
\end{align*}
and thus, the claim follows.
\end{proof}
\noindent Pulling back linear functionals on $\Ind_{P_{n}}^{G_{2n}}\pi$ along $\delta$ yields $E$-valued distributions on $G_n$, i.e.~given $\lambda\colon \Ind_{P_{n}}^{G_{2n}}\pi\to E$ we define
\begin{align}\label{pullbackdist}
\mu_{\lambda}:=\lambda\circ\delta\in \Dist(G_n,E).
\end{align}
For every element $\varphi\in\Ind_{P_{n}}^{G_{2n}}\pi$ we let $\xi^{\lambda}_{\varphi}\colon G_{2n}\to E$ be the function given by $\xi^{\lambda}_{\varphi}(g)=\lambda(g\varphi)$.
If $\mathcal{C}$ is a compact, open subgroup of $G_{n}$, we put $\xi^{\lambda}_{\mathcal{C}}=\xi^{\lambda}_{\delta(\cf_{\mathcal{C}})}$.
\begin{Lem}\label{density}
Let $\lambda\colon \Ind_{P_{n}}^{G_{2n}}\pi\to E$ be a linear functional and $\mathcal{C}\subset K_{n}$ a compact, open subgroup.
Then for all $f\in C^{0}_{c}(G_{n},E)$, which are $\mathcal{C}$-invariant under right multiplication, we have
\begin{align*}
\int_{G_{n}}{f(g)\mu_{\lambda}(dg)}=[K_{n}\colon \mathcal{C}] \int_{G_{n}}{f(g)\ \xi^{\lambda}_{\mathcal{C}}\left(\begin{pmatrix}g&0\\0&1 \end{pmatrix}\right)d^{\ast}g}.
\end{align*}
\end{Lem}
\begin{proof}
It is enough to prove the formula in the case $f=\cf_{A\mathcal{C}}=(A,1_n)\cf_{\mathcal{C}}$ with $A\in G_{n}$.
In this case we have
\begin{align*}
\int_{G_{n}}{f(g)\mu_{\lambda}(dg)}
&=\int_{G_{n}}{(A,1_n)\cf_{\mathcal{C}}(g)\mu_{\lambda}(dg)}\\
&=\xi^{\lambda}_{\mathcal{C}}\left(\begin{pmatrix}A&0\\0&1_n\end{pmatrix}\right)
=[K_{n}\colon \mathcal{C}]\int_{G_{n}} \cf_{\mathcal{C}}(g)\ \xi^{\lambda}_{\mathcal{C}}\left(\begin{pmatrix}Ag&0\\0&1_n\end{pmatrix}\right) d^{\ast}g\\
&=[K_{n}\colon \mathcal{C}]\int_{G_{n}} f(g)\ \xi^{\lambda}_{\mathcal{C}}\left(\begin{pmatrix}g&0\\0&1_n\end{pmatrix}\right) d^{\ast}g,
\end{align*}
which is exactly what we claimed.
\end{proof}
Besides the multiplicative equivariance properties of the map $\delta$ there is an additional additive equivariance.
We let $M_{n}(F)$ act on $C^{0}_{c}(M_{n}(F),E)$ by
$$\left(X\star f\right)(g)=f(g+X)$$ and on $\Ind_{P_n}^{\GL_{2n}}\pi$ via the embedding
\begin{align*}
M_{n}(F)\to G_{2n},
\ X\mapsto \begin{pmatrix}1_n&X\\0&1_n\end{pmatrix}.
\end{align*}

\begin{Lem}\label{addequi}
Let $A\in G_{n}$ be a matrix and $\mathcal{D}\subset G_{n}$ a compact, open subset such that $\rho$ is stable under $\left\{1_n\right\}\times \mathcal{D}$. Then we have
\begin{align*}
\delta(X\star f)= \begin{pmatrix}1_n&X\\0&1_n\end{pmatrix} \delta(f).
\end{align*}
for all $f\in C^{0}(A\mathcal{D},E)$ and all matrices $X\in M_{n}(F)$ with $A\mathcal{D}+X=A\mathcal{D}$.
\end{Lem}
\begin{proof}
It is convenient to introduce an untwisted version of the map $\delta$.
For every $\rho^{\prime}\in \pi$ we define the $M_{n}(F)$-equivariant map
\begin{align*}
\partial_{\rho^{\prime}}\colon C_{c}(M_{n}(F),E)\too \Ind_{P_n}^{\GL_{2n}}\pi
\end{align*}
as follows: if $g\in G_{2n}$ is of the form
\begin{align*}
g=\begin{pmatrix}g_{1}&\ast\\0&g_{2}\end{pmatrix} \begin{pmatrix}0&1_n\\ 1_n&0\end{pmatrix} \begin{pmatrix}1_n&u\\0&1_n\end{pmatrix}
\end{align*}
with $g_{1},g_2\in G_{n}$ and $u\in M_{n}(F)$, we put
$$\partial_{\rho^{\prime}}(f)(g)=f(u)\cdot \pi\left(g_{1},g_{2}\right)\rho^{\prime}$$
and otherwise we set $\partial_{\rho^{\prime}}(f)(g)=0$.

Let $f$ be a function in $C^{0}(A\mathcal{D},E)$.
Then by assumption we have
\begin{align*}
\delta_{\rho}(X\star f)
&=\partial_{(1,A)\rho}(X\star f)\\
&=\begin{pmatrix}1_n&X\\0&1_n\end{pmatrix} \partial_{(1,A)\rho}(f)\\
&=\begin{pmatrix}1_n&X\\0&1_n\end{pmatrix} \delta_{\rho}(f),
\end{align*}
which proves the assertion.
\end{proof}
\end{subsection}
\begin{subsection}{Weakly ordinary stabilizations}\label{int}
Using the map discussed in the previous section we want to construct maps from function spaces to irreducible locally algebraic representations of $G_{2n}$. The main aim of this section is to give a criterion on when these maps respect integral structures.

Let $E$ be a field of characteristic $0$ and $R\subset E$ a subring with field of fractions $E$, which is integrally closed in $E$. We are mostly interested in the case that $p$ is not invertible in $R$. 
\begin{Def}
Let $V\in \mathfrak{C}^{\sm}_{E}(G_{2n})$ be an irreducible representation.
\begin{enumerate}[(i)]
 \item A \textit{stabilization} $\Theta=(\pi,\rho,\vartheta)$ of $V$ consists of
\begin{itemize}
\item an irreducible representation $\pi\in\mathfrak{C}^{\sm}_{E}(G_n\times G_n)$,
\item a non-zero element $\rho\in\pi$ and
\item a non-zero $G$-equivariant homomorphism $\vartheta\colon \Ind_{P_{n}}^{G_{2n}}\pi\to V$.
\end{itemize}
 \item Let $\Theta=(\pi,\rho,\vartheta)$ be a stabilization of $V$.
Write $\pi$ as a tensor product $\pi_{1}\otimes \pi_{2}$ of irreducible representations $\pi_1,\pi_2\in \mathfrak{C}^{\sm}_{E}(G_{n})$.
We put $\alpha:=\alpha_\Theta:=\omega_{2}(\varpi),$ where $\omega_{2}$ denotes the central character of $\pi_{2}$.
The stabilization $\Theta$ is called $R$-\textit{integral} if $\alpha^{-1}\in R$.
\end{enumerate}
\end{Def}
\begin{Rem}
\begin{enumerate}[(i)]\thmenumhspace
\item The exact value of $\alpha$ does depend on the choice of uniformizer $\varpi$ if $\omega_{2}$ is ramified. But changing the uniformizer changes $\alpha$ only by a root of unity. Hence, the integrality condition for stabilizations is independent of the choice of uniformizer.
\item If $n=1$, the existence of a stabilization is equivalent to $V$ not being supercuspidal.
\end{enumerate}
\end{Rem}
\begin{Exa}\label{example}
\begin{enumerate}[(i)]\thmenumhspace
   \item Our guiding example will be the case of unramified principal series representations. Let $\chi_1,\ldots,\chi_r\colon F^{\ast}\to E^{\ast}$ be unramified characters. They induce a character
$$
\chi\colon B_r\too E^{\ast}\quad\mbox{ via }\quad b\mapstoo \prod_{i=1}^{r}\chi_i(b_{ii}).
$$
We will write $\Ind_{B_r}^{G_r}(\chi_1,\ldots,\chi_r)$ for the smooth representation $\Ind_{B_r}^{G_r}\chi$. Now, assume that $r=2n$ and that $V=\Ind_{B_{2n}}^{G_{2n}}(\chi_1,\ldots,\chi_{2n})$ is irreducible. We will call $\Theta^{\ur}=(\pi^{\ur},\rho^{\ur},\vartheta^{\ur})$ with
\begin{itemize}
\item $\pi^{\ur}=\Ind_{B_n}^{G_n}(\chi_1,\ldots,\chi_n)\otimes\Ind_{B_n}^{G_n}(\chi_{n+1},\ldots,\chi_{2n})$
\item $\rho^{\ur}$ the unique normalized $K_n\times K_n$-fixed vector in $\pi$
\item $\vartheta^{\ur}$ the canonical isomorphism
\end{itemize} 
the unramified stabilization of $V$ with respect to $(\chi_1\ldots \chi_{2n})$.
The unramified stabilization is integral if $\alpha^{-1}=\prod_{i=n+1}^{2n}\chi_i(\varpi)^{-1}\in R$.
Note that we get different models of $V$ and hence different unramified stabilizations by reordering (and normalizing) the characters $\chi_i$.
This amounts to $\binom{2n}{n}$ different unramified stabilizations for each irreducible unramified principal series representation.
Each of these stabilizations is defined over a finite extension of the field of definition of $V$, which can be made explicit in terms of Hecke eigenvalues.
   \item For $r\geq 1$ let $\St_r$ denote the Steinberg representation of $G_{r}$. Then $\St_{2n}$ has a canonical stabilization of the form $\Theta^{\St}=(\St_{n}\otimes \St_{n},\rho\otimes\rho,\vartheta)$, where $\rho\in \St_{n}$ is the normalized Iwahori-fixed vector. The Steinberg stabilization is defined over $\Q$ and is $\Z$-integral since $\alpha=1$.
	\item Assume we have given an irreducible representation $V\in \mathfrak{C}_{E}(G_{2n})$ together with a stabilization $\Theta=(\pi,\rho,\vartheta)$. Let $\chi\colon F^{\ast}\to E^{\ast}$ be a continuous finite order character. After choosing a non-zero element $e\in E[\chi]$, we can define the twisted stabilization $\Theta\otimes\chi=(\pi\otimes\chi,\rho\otimes e,\vartheta_{\chi}))$ of $V\otimes \chi$, where $\vartheta_{\chi}$ is given by the composition
\begin{align*}
\Ind_{P_{n}}^{G_{2n}}\pi\otimes \chi\xrightarrow{\cong} \Ind_{P_{n}}^{G_{2n}}\pi\otimes \chi \xrightarrow{\vartheta\otimes \id} V\otimes\chi.
\end{align*}	
	The twisted stabilization $\Theta\otimes\chi$ is weakly ordinary if and only if $\Theta$ is weakly ordinary.
\end{enumerate}
\end{Exa}
\noindent Given a stabilization $\Theta=(\pi,\rho,\vartheta)$ of $V$ we can precompose $\vartheta$ with the map $\delta_{\rho}$ of the previous section to obtain
\begin{align*}
\delta_{\Theta}=\vartheta \circ \delta_\rho \colon C^{0}_{c}(G_n,E)\too V.
\end{align*}
The following lemma explains the notion of integrality of stabilizations. It (resp.~its locally algebraic counterpart below) is one of the main ingredients to prove that the distributions we construct in Section \ref{global} are bounded.
\begin{Lem}\label{integrality}
Let $V\in \mathfrak{C}^{\sm}_{E}(G_{2n})$ be irreducible and $R$-integral and let $L\subset V$ be a $G_{2n}$-stable lattice. If $\Theta=(\pi,\rho,\vartheta)$ is an $R$-integral stabilization of $V$, then there exists a non-zero constant $c\in E^{\ast}$ such that
\begin{align*}
c\cdot\delta_{\Theta}\colon C_{c}^{0}(G_n,R) \too L \subset V.
\end{align*}
\end{Lem}
\begin{proof}
In fact, we prove a stronger statement: Let $Q\subset G_{2n}$ be the subgroup
\begin{align*}
Q=\left\{\begin{pmatrix}g&u\\0&1_n\end{pmatrix}\in G_{2n}\mid g\in G_{n},\  u\in M_{n}(F)\right\}
\end{align*}
and let $m\geq 1$ be a natural number such that $\left\{1_n\right\}\times K_{n}^{(m)}$ is in the stabilizer of $\rho$. Then the image of $C_c^{0}(G_{n},R)$ under $\delta=\delta_{\rho}$ is contained in the $R[Q]$-module generated by $\delta(\cf_{K_{n}^{(m)}})$.

\noindent For this let $A\in Q$ be the matrix given by
\begin{align*}
A=\begin{pmatrix}
\varpi 1_n & (\varpi-1)1_n\\
0 & 1_n
\end{pmatrix}.
\end{align*}
It is the product of the two matrices
\begin{align*}
A_{0}=\begin{pmatrix}
1_n & (\varpi-1)1_n\\
0 & 1_n
\end{pmatrix}
\  \mbox{ and  }\   
A_{1}=\begin{pmatrix}
\varpi 1_n & 0\\
0 & 1_n
\end{pmatrix}.
\end{align*}
\noindent In the following, we will abbreviate a scalar matrix with $a\in F$ on its diagonal simply by $a$. Using Lemma \ref{equi} we get
\begin{align*}
A\delta(\cf_{K_{n}^{(r)}})(g)&=A_{0}A_{1}\delta(\cf_{K_{n}^{(r)}})(g)\\
&=A_{0}\delta(A_{1}\cf_{K_{n}^{(r)}})(g)\\
&=A_{0}\delta(\cf_{\varpi K_{n}^{(r)}})(g)\\
&=\cf_{\varpi K_{n}^{(r)}}(u+\varpi-1)\cdot\pi(g_{1},g_{2}(u+\varpi-1))\rho\\
&=\cf_{\varpi K_{n}^{(r)}}(u+\varpi-1)\cdot \alpha\cdot \pi(g_{1},g_{2})\rho\\
&=\alpha\cdot \cf_{K_{n}^{(r+1)}}(u)\cdot \pi(g_{1},g_{2})\rho\\
&=\alpha\cdot\cf_{K_{n}^{(r+1)}}(u)\cdot \pi(g_{1},g_{2}u)\rho\\
&=\alpha\cdot\delta(\cf_{K_{n}^{(r+1)}})(g)
\end{align*}
for all $r\geq m$ and $g=\begin{pmatrix}g_{1}&\ast\\ 0&g_{2}\end{pmatrix} \begin{pmatrix}0&1_n\\ 1_n&0\end{pmatrix} \begin{pmatrix}1_n&u\\ 0&1_n\end{pmatrix}\in G_{2n}$ with $g_{i},u\in G_{n}$. Therefore, the claim follows by induction.
\end{proof}

In the remainder of this subsection we want to discuss locally algebraic versions of the previous results. In particular, $E\subset \C_{p}$ will be a finite extension of $\Q_{p}$ with valuation ring $R$. We assume that every embedding $\sigma\colon F\to \C_{p}$ factors through $E$. Let $\Z[\Hom(F,E)]$ be the free abelian group on the set of field homomorphisms from $F$ to $E$. We identify $$a=\sum a_{\sigma}\sigma \in \Z[\Hom(F,E)]$$ with the group homomorphism
$$a\colon F^{\ast}\to E^{\ast},\quad x\mapsto \prod_{\sigma}\sigma(x)^{a_{\sigma}}.$$
Given $a=\sum a_{\sigma}\sigma$ and $b=\sum b_{\sigma}\sigma$ we write $a\leq b$ if and only if $a_\sigma \leq b_\sigma$ for all $\sigma$.

Further, we fix an irreducible smooth representation $V_{\sm}\in\mathfrak{C}^{\sm}_{E}(G_{2n})$ and an irreducible, finite-dimensional $\Q_{p}$-rational representation $V_{\al}$ of $\mathbb{GL}_{2n,F}$ on a finite dimensional $E$-vector space. By definition there exist irreducible $E$-rational representations $V_{\sigma}$ of $\mathbb{GL}_{2n,E}$ for all embeddings $\sigma\colon F \to E$ and a $G$-equivariant isomorphism
\begin{align*}
V_{\al}\cong \bigotimes_{\sigma\colon F\to E}V_{\sigma}.
\end{align*}
Here $G$ acts on $V_{\sigma}$ through the embedding $\sigma\colon F\to E$. We will denote the highest weight of $V_{\sigma}$ with respect to the Borel subgroup of upper triangular matrices by $\mu_{\sigma}=(\mu_{\sigma,1},\ldots, \mu_{\sigma,2n})$. We set $e_{\sigma}=\mu_{\sigma,1}+\ldots+\mu_{\sigma,n}$ and define
\begin{align*}
e_{\al}=\sum_{\sigma \in \Hom(F,E)} e_{\sigma}\sigma.
\end{align*}
\begin{Def} Let $V=V_{\sm}\otimes V_{\al}$ be as above.
\begin{enumerate}[(i)]
    \item A \textit{stabilization} $\Theta$ of $V$ is just a stabilization of its smooth part $V_{\sm}$.
    \item We say that a stabilization $\Theta$ of $V$ is \textit{weakly ordinary} (with respect to $V_{\al}$) if $\alpha_{\Theta}^{-1}e_{\al}(\varpi)^{-1}\in R$.
		\item A \textit{critical point} $V_s$ of $V$ is a one-dimensional $G_n\times G_n$-subrepresentation of $V_{\al}$.
\end{enumerate}
\end{Def}
\begin{Rem}
\begin{enumerate}[(i)]\thmenumhspace
	\item The notion of weak ordinarity can be seen as an automorphic version of Panchishkin's $p$-ordinarity condition of motives (see Section 5 of \cite{Pa}). 
	\item Every irreducible unramified principal series representation, which is ordinary in the sense of Definition \ref{orddef}, has a weakly ordinary stabilization.
	\item On the notion of critical points: We will see in Section \ref{cohom} that critical points of L-functions of global automorphic representations correspond to certain one-dimensional subrepresentations of $V_{\al}$.
\end{enumerate}
\end{Rem}
\noindent Given a stabilization $\Theta$ and a critical point $V_s\subset V_{\al}$ of $V$ we define
\begin{align*}
\delta_{\Theta,s}\colon C_{c}^{0}(G_{n},E)\otimes V_s \xrightarrow{\delta_\vartheta \otimes \id} V_{\sm}\otimes V_s\hookrightarrow V.
\end{align*}
As an immediate consequence of Lemma \ref{equi} we get
\begin{Lem} 
The map $\delta_{\Theta,s}$ is $G_{n}\times \mathcal{D}$-equivariant, where $\mathcal{D}\subset G_n$ is the open subgroup given by $\mathcal{D}=\left\{g\in G_n\mid (g,g)\rho=\rho\right\}$.
\end{Lem}
\noindent There is also a version of the integrality Lemma \ref{integrality} in this setup. Let $\chi_{s}$ be the character of the one-dimensional $G_n\times G_n$-representation $V_{s}$. After choosing an isomorphism $E[\chi_{s}]\cong V_{s}$ we can consider the lattice $$C_c^{0}(G,R)\otimes_{R} \chi_{s}\subset C_{c}^{0}(G_{n},E)\otimes V_s$$
as defined in \eqref{characterspace}.
\begin{Lem}\label{integrality 2}
Suppose $V$ is $R$-integral and let $L\subset V$ be a $G_{2n}$-lattice. If $\Theta$ is a weakly ordinary stabilization of $V$ and $V_s$ is a critical point of $V$, then there exists a non-zero constant $c\in E^{\ast}$ such that
\begin{align*}
c\cdot\delta_{\Theta,s}\colon C_{c}^{0}(G_n,R)\otimes_R \chi_{\Theta}\to L \subset V.
\end{align*}
\end{Lem}
\begin{proof}
Let $u \in G_{2n}$ be the matrix
\begin{align*}
u=\begin{pmatrix}
1_n&  -1_n \\
0 & 1_n
\end{pmatrix}
\end{align*}
and $\nu\colon\mathbb{G}_{m}\to\mathbb{GL}_{2n}$ the cocharacter which sends $t\in\mathbb{G}_{m}$ to the diagonal matrix $\nu(t)$ with $\nu(t)_{ii}=t$ if $1\leq i\leq n$ and $\nu(t)_{ii}=1$ if $n+1\leq i\leq 2n$. We put $\nu^{\prime}=u\nu u^{-1}$. Then the matrix $A$ we considered in the proof of Lemma \ref{integrality} is nothing but $\nu^{\prime}(\varpi)$. We can view $V_{\al}$ as a $\Res_{F/\Q_p}(\mathbb{G}_{m,F})$-representation via $\mu^{\prime}$ and hence, it has a weight space decomposition, i.e.~there exists a basis $(v_{1},\ldots,v_{k})$ of $V_{\al}$ and elements $e_{1},\ldots,e_{k}\in\Z[\Hom(F,E)]$ such that
\begin{align*}
Av_{l}=e_{l}(\varpi)v_{l}\quad \forall 1\leq l\leq k.
\end{align*}
From the proof of Lemma \ref{integrality} we see that there exists an integer $m\geq 1$ such that
$$\delta_{\Theta}(\cf_{K_{n}^{(r+1)}})\otimes v_l=
\alpha_{\Theta}^{-1}e_{l}(\varpi)^{-1}\cdot A\left(\delta_{\Theta}(\cf_{K_{n}^{(r)}})\otimes v_{l}\right)$$
for $r\geq m$ and all $1\leq l \leq k$. After multiplication with a non-zero constant we might assume that $\delta_{\Theta}(\cf_{K_{n}^{(m)}})\otimes v_{l}\in L$ for all $1\leq l\leq k$.
By definition we have $e_{l}\leq e_{\al}$ for all $1\leq l\leq k$. Thus, by the ordinarity assumption on $\Theta$ we inductively get
$$\delta_{\Theta}(\cf_{K_{n}^{(r+1)}})\otimes v_l \in L$$
for all $r\geq m$.
Multiplying the $v_{i}$ with appropriate non-zero constants we can assume that $R[\chi_{s}]\subset V_{s}$ is a submodule of the $R$-span of $v_{1},\ldots,v_{k}$ and the claim follows.
\end{proof}
\noindent For the sake of clarity let us work out the conditions of the preceding lemma in the case $F=\Q_{p}$, $n=1$, $V_{\al}=\Sym^{k}(\Q_{p}^{2})^{\vee}$ and $V_{\sm}=\Ind_{B_{2}}^{G_{2}}(\chi_{1},\chi_{2})$ an irreducible principal series representation. We set $\alpha=\chi_{2}(p)$ and $\alpha^{\prime}=\chi_{1}(p)$.
The highest weight of $V_{\al}$ is given by $\mu=(0,-k)$.
The existence of a lattice in $V=V_{\sm}\otimes V_{\al}$ implies that
\begin{enumerate}[(i)]
  \item $\alpha\alpha^{\prime} p^{-k}\in R^{\ast}$, \label{con}
	\item $\alpha\in R$ and $p\alpha^{\prime} \in R$.
\end{enumerate}
The associated unramified stabilization is ordinary if and only if $\alpha^{-1}\in R$. So the weak ordinarity hypothesis together with the existence of a lattice implies that $\alpha$ is a unit in $R$. Vice versa, it is easy to see that the representation $V$ has a lattice if $\alpha\in R^{\ast}$ and condition \eqref{con} holds.
\end{subsection}
\begin{subsection}{Local Shalika models and local distributions}\label{localshalika}
The Shalika subgroup $S$ of $G_{2n}$ is given by
\begin{align*}
S=\left\{\left.\begin{pmatrix}h&0\\0&h\end{pmatrix}\begin{pmatrix}1_n&X\\0&1_n\end{pmatrix}\right| h\in G_{n}, X\in M_{n}(F)\right\}.
\end{align*}
We fix a locally constant character $\eta\colon F^{\ast}\to \C^{\ast}$.
It induces a character $\eta\psi\colon S\to \C^{\ast}$ via
\begin{align*}
\begin{pmatrix}h&0\\0&h\end{pmatrix}\begin{pmatrix}1_n&X\\0&1_n\end{pmatrix}\mapsto \eta(\det(h))\psi(\tr(X)).
\end{align*}
\begin{Def}\label{shalika}
An irreducible representation $V\in\mathfrak{C}^{\sm}_{\C}(G_{2n})$ has a (\textit{local}) $(\eta,\psi)$-\textit{Shalika model} if there exists a non-zero functional $\lambda\colon V\to\C$ such that
$$\lambda(s\varphi)=\eta\psi(s)\lambda(\varphi)\quad \forall s\in S,\ \varphi\in\pi.$$
The functional $\lambda$ is called a (\textit{local}) $(\eta,\psi)$-\textit{Shalika functional}.
\end{Def}
\begin{Rem}\label{twistSha}
\begin{enumerate}[(i)]\thmenumhspace
\item\label{twistSha1} Suppose $V$ has an $(\eta,\psi)$-Shalika functional $\lambda$.
Let $\chi\colon F^{\ast}\to \C^{\ast}$ be a locally constant character and $e\in \C[\chi]$ a non-zero element. Then
\begin{align*}
\lambda_{\chi}\colon V\otimes \C[\chi]\to \C,\quad v\otimes e\mapsto\lambda (v)
\end{align*}
defines a $(\eta\chi^{2},\psi)$-Shalika functional on $V\otimes\chi$.
\item If $\eta$ is the trivial character, Jacquet and Rallis have shown in \cite{JR} that Shalika functionals are - if they exist - unique up to multiplication by a constant.
An elementary proof of this fact can be found in \cite{N}.
Using the first remark one gets the uniqueness of $(\eta,\psi)$-Shalika functionals if $\eta$ is a square. 
\item If $\eta$ is a finite order character, Ash and Ginzburg have proven the uniqueness of Shalika functionals for unramified, irreducible principal series representations under a technical condition on the induction parameter (see Lemma 1.7 of \cite{AG}).
\end{enumerate}
\end{Rem}
\noindent In view of the above remarks we make the following
\begin{Ass}
We assume that local Shalika functionals are - if they exist - unique up to multiplication by a non-zero scalar.
\end{Ass}
\noindent Suppose we have given a stabilization $\Theta=(\pi,\rho,\vartheta)$ of an irreducible representation $V\in\mathfrak{C}^{\sm}_{\C}(G_{2n})$, which has a Shalika functional $\lambda$.
Let $\mu_{\Theta}:=\mu_{\lambda\circ\vartheta}\in \Dist(G_n,\C)$ be the distribution defined in \eqref{pullbackdist}.
\begin{Lem}\label{localdst}
Assume that $V$ is generic and that there exists $t\in\C$ such that $V\otimes\left|\det\right|^{t}$ is unitary.
For every continuous character $\chi\colon F^{\ast}\to \C^{\ast}$ the integral
\begin{align}\label{localdist}
E(\Theta,\chi,s):=\int_{G_n}\chi(\det(g))\left|\det(g)\right|^{s-1/2}\ d\mu_{\Theta}(g)
\end{align}
converges absolutely for $\mathfrak{Re}(s)$ large.
There is a factorization
\begin{align}\label{localmod}
E(\Theta,\chi,s)=e(\Theta,\chi,s)\cdot L(V\otimes\chi,s),
\end{align}
where $e(\Theta,\chi,s)$ is an entire function.
Hence, $E(\Theta,\chi,s)$ can be extended to a meromorphic function on $\C$.
\end{Lem}
\begin{proof}
Let $\mathcal{C}\subseteq K_n$ be an open subgroup, which is contained in the kernel of $\chi\circ\det$. By Lemma \ref{density} we have the following equality:
$$E(\Theta,\chi,s)=[K_{n}\colon \mathcal{C}] \int_{G_{n}}{\chi(\det(g))\left|\det(g)\right|^{s-1/2}\ \xi^{\lambda\circ \vartheta}_{\mathcal{C}}\left(\begin{pmatrix}g&0\\0&1_n \end{pmatrix}\right)d^{\ast}g}$$
The function $\xi^{\lambda\circ\vartheta}_{\mathcal{C}}$ is an element of the Shalika model of $V$. Hence, the claim follows from \cite{FJ} Proposition 3.1.
\end{proof}
\begin{Rem}
In the case $n=1$, the vector $\rho$ of the stabilization $\Theta$ is determined uniquely up to a constant and therefore, the modified Euler factor essentially does not depend on $\rho$.
If $n>1$, there are different choices of $\rho$ yielding a priori different modified Euler factors.
\end{Rem}

The modified Euler factors $E(\Theta,\chi,s)$ behave well under twisting.
Using Remark \ref{twistSha} \eqref{twistSha1} a straightforward calculation gives
\begin{Lem}\label{dothetwist}
Let $\chi^{\prime}\colon F^{\ast}\to \C^{\ast}$ be a continuous character.
Then the equality
\begin{align*}
E(\Theta\otimes \chi^{\prime},\chi,s)=E(\Theta,\chi^{\prime}\chi,s)
\end{align*}
holds.
\end{Lem}
Let us take a closer look at the spherical example:
Fix unramified characters $\chi_{1},\ldots,\chi_{2n}$ such that $V=\Ind_{B_{2n}}^{G_{2n}}(\chi_1,\ldots,\chi_{2n})$ is irreducible and has a unitary twist.
Assume that $V$ has a $(\eta,\psi)$-Shalika model.
Then by Proposition 1.3 of \cite{AG} we know that $\eta$ is unramified and we may assume that $\chi_{i}=\eta\chi_{2n-i+1}^{-1}$ (which we will do in the following).
Conversely every such unramified principal series representation has a Shalika model.
More precisely: Write $\beta_{i}=\chi_{i}(\varpi)q^{n-i+1/2}=\alpha_{i}q^{n-i+1/2}$ for the Satake parameters of $V$ and let $\left|\cdot\right|_{\infty}$ be the standard norm on $\C$.
If we assume that $\left|\beta_{i}\beta_{j}\right|_{\infty}<1$ for all $1\leq i<j \leq n$, then by \cite{AG}, Lemma 1.4, the following absolutely convergent integral gives the Shalika functional: 
\smaller\begin{align*}
\lambda(\varphi)=\int_{K_{n}}\int_{M_{n}(F)}{\varphi
\left(
\begin{pmatrix} 0&1_n\\1_n&0 \end{pmatrix}
\begin{pmatrix} g&0\\0&g \end{pmatrix}
\begin{pmatrix} 1_n&X\\0&1_n \end{pmatrix}
\right)
\eta^{-1}(\det(g))\psi^{-1}(\tr(X))\ dXd^{\ast}g}.
\end{align*}\larger
Here $dX$ denotes an additive Haar measure on $M_{n}(F)$.
If $\beta_{i}\beta_{j}\neq \eta^{\pm 1}(\varpi)$ for all $1\leq i<j\leq n$, then the Shalika functional can be defined via analytic continuation of the above integral (see the proof of \cite{AG}, Proposition 1.3).

Let $\Theta^{\ur}=(\pi^{\ur},\rho^{\ur},\vartheta^{\ur})$ be the unramified stabilization of $V$ with respect to $(\chi_1,\ldots \chi_{2n})$.
We can write $\rho^{\ur}=\rho_1\otimes \rho_2$, where $\rho_{1}\in \Ind_{B_{n}}^{G_{n}}(\chi_{1},\ldots,\chi_{n})$ and $\rho_{2}\in \Ind_{B_{n}}^{G_{n}}(\chi_{n+1},\ldots,\chi_{2n})$ are normalized such that $\rho_{i}(k)=1$ for $i=1,2$ and all $k\in K_{n}$.
\begin{Pro}\label{integrals}
Let $V=\Ind_{B_{2n}}^{G_{2n}}(\chi_1,\ldots,\chi_{2n})$ be an irreducible unramified principal series as above with Shalika functional $\lambda$.
Assume that $\beta_{i}\beta_{j}\neq \eta^{\pm 1}(\varpi)$ for all $1\leq i<j\leq n$.
Then we have
\begin{align}\label{prevaluation}
\int_{G_{n}}f(g) \mu_{\Theta^{\ur}}(dg)= \int_{M_{n}(F)}{\cf_{G_{n}}(X)\rho_{2}(X)f(X)\psi^{-1}(\tr(X))\ dX}
\end{align}
for all $f\in C_{c}^{0}(G_{n},\C)$ which are invariant under conjugation by $K_{n}$. 

\noindent As a special case we get:
For every continuous character $\chi\colon F^{\ast}\to \C^{\ast}$ we have
\begin{align}\label{evaluation}
E(\Theta^{\ur},\chi,s)= \int_{M_{n}(F)}{\cf_{G_{n}}(X)\rho_{2}(X)\chi(\det(X))\left|\det(X)\right|^{s-1/2}\psi^{-1}(\tr(X))\ dX}
\end{align}
for $\mathfrak{Re}(s)$ large.
\end{Pro}
\begin{proof}
For every $s\in \C$ we define
\begin{align*}
V^{s}=\Ind_{B_{n}}^{G_{n}}(\chi_{1}\left|\cdot\right|^{s},\ldots,\chi_{2n}\left|\cdot\right|^{s})
\end{align*}
and let
\begin{align*}
\rho^{s}_{1}\in \Ind_{B_{n}}^{G_{n}}(\chi_{1}\left|\cdot\right|^{s},\ldots,\chi_{n}\left|\cdot\right|^{s}) \mbox{ resp. } \rho^{s}_{2}\in \Ind_{B_{n}}^{G_{n}}(\chi_{n+1}\left|\cdot\right|^{-s},\ldots,\chi_{2n}\left|\cdot\right|^{-s})
\end{align*} be the normalized spherical vectors, i.e.~$\rho^{s}_{i}(k)=1$ for all $k\in K_{n}$, $i=1,2$.
We write $\delta^{s}$ for the corresponding maps
\begin{align*}\delta^{s}\colon C_{c}^{0}(G_{n},\C)\to\Ind_{B_{2n}}^{G_{2n}}(\chi_{1}\left|\cdot\right|^{s},\ldots,\chi_{n}\left|\cdot\right|^{s},\chi_{n+1}\left|\cdot\right|^{-s},\ldots,\chi_{2n}\left|\cdot\right|^{-s})=:\pi^{s}
\end{align*}
and $\lambda^{s}$ for the Shalika functional of $V^{s}$.
Since the map $s\mapsto \lambda^{s}(\delta^{s}(f))$ is analytic we can compute the left hand side of \eqref{prevaluation} as the analytic continuation to $s=0$ of the integral 
\begin{align*}
&\int_{K_{n}}\int_{M_{n}(F)}{\hspace{-1,0em}\delta^{s}(f)
\left(
\begin{pmatrix} 0&1_n\\1_n&0 \end{pmatrix}
\begin{pmatrix} g&0\\0&g \end{pmatrix}
\begin{pmatrix} 1_n&X\\0&1_n \end{pmatrix}
\right)
\eta^{-1}(\det(g))\psi^{-1}(\tr(X))\ dXd^{\ast}g} \\
=&\int_{K_{n}}\int_{G_{n}}{
\rho^{s}_{1}(g)\rho^{s}_{2}(gX)f(X)
\eta^{-1}(\det(g))\psi^{-1}(\tr(X))\ dXd^{\ast}g}\\\\
=&\int_{K_{n}}\int_{G_{n}}{
\rho^{s}_{1}(g)\rho^{s}_{2}(Xg)f(g^{-1}Xg)
\eta^{-1}(\det(g))\psi^{-1}(\tr(g^{-1}Xg))\ dXd^{\ast}g}\\
=&\int_{G_{n}}{
\rho_{2}(X)\left|\det(X)\right|^{-s}f(x)
\psi^{-1}(\tr(X))\ dX}.
\end{align*}
The claim follows since $f$ has compact support inside $G_n$.
\end{proof}

\end{subsection}
\begin{subsection}{Computation of modified Euler factors}
We are going to evaluate the Euler factors $E(\Theta^{\ur},\chi,s)$ of Proposition \ref{integrals}.
Using the Iwasawa decomposition we can reduce the integral over $G_n$ to sums of integrals over explicit compact open subsets of $G_n$.
Most of these integrals vanish by orthogonality of characters applied either to the fixed additive character $\psi$ or the multiplicative character $\chi$. 

A fair amount of the computations work in a more general setup:
We fix an irreducible representation $V\in \mathfrak{C}_\C(G_{2n})$, which admits a Shalika functional $\lambda$ and a stabilization $\Theta=(\pi,\rho,\vartheta)$.
Let $I_{n}^{(r)}\subset K_n$ denote the Iwahori subgroup of Level $\mathfrak{p}^{r}$, i.e.~the set of all matrices in $K_n$, which are upper triangular modulo $\mathfrak{p}^{r}$ and write $\widetilde{I}_{n}^{(m)}\subset I_{n}^{(m)}$ for the subgroup of matrices which are unipotent upper triangular modulo $\p^{m}$.
We assume that $\rho$ is stabilized by the group $\left\{1_n\right\}\times I_n^{(1)}$.
In particular, $\alpha_{\Theta}$ is independent of the choice of a local uniformizer.
We are going to use the following two properties:
\begin{enumerate}[(I)]
  \item $K_{n}^{(m)}\subset I_{n}^{(1)}$ for all $m\geq 1$ and \label{prop1}
	\item $\det\colon I_{n}^{(1)}\to U$ is surjective. \label{prop2}
\end{enumerate}

\begin{Def}
The \textit{order} $\ord(A)$ of a matrix $A\in M_{n}(F)$ is the minimum of the $\nu(A_{ij})$, $1\leq i,j\leq n$.
\end{Def}
\noindent It is a straightforward calculation to show that 
\begin{align*}
\ord(AB)\geq\ord(A)+\ord(B)
\end{align*}
for all $A,B\in M_{n}(F)$.
In particular, we get an equality if one of the matrices is in $K_{n}$.

\begin{Lem}\label{vanish}
Let $A\in G_{n}$ be a matrix and $m\in\Z$ an integer with $1\leq m<-\ord(A)$. We have the following equality:
\begin{align*}
\int_{G_n}{\cf_{AK_{n}^{(m)}}\ d\mu_{\Theta}}=0.
\end{align*}
\end{Lem}
\begin{proof}
Choose $k,l\in\left\{1,\ldots,n\right\}$ such that $\ord(A)=\nu(A_{kl})$.
By assumption, there exists $b\in F^{\ast}$ with $\nu(b)=-\nu(A_{kl})-1$ and $\psi(A_{kl}b)\neq 1$.
Define the matrix $B\in \varpi^{m}M_{n}(\mathcal{O})$ via
\begin{align*}
B_{ij}=\begin{cases}
b & \mbox{if } i=l,\ j=k,\\
0 & \mbox{else}. 
\end{cases}
\end{align*}
The indicator function on the set $AK_{n}^{(m)}$ is clearly invariant under addition by matrices in $A\varpi^{m}M_{n}(\mathcal{O})$.
Hence, by Lemma \ref{addequi} and property \eqref{prop1} we get
\begin{align*}
\int_{G_n}{\cf_{AK_{n}^{m}}\ d\mu_{\Theta}}
&= \lambda (\delta_{\Theta}(\cf_{AK_{n}^{(m)}}))\\ 
&= \lambda (\delta_{\Theta}(AB\star\cf_{AK_{n}^{(m)}}))\\
&= \lambda \left(
\begin{pmatrix}
1_n & AB\\
0 & 1_n
\end{pmatrix}
\delta_{\Theta}(\cf_{AK_{n}^{(m)}})\right)\\
&= \psi(\tr(AB)) \lambda (\delta_{\Theta}(\cf_{AK_{n}^{(m)}}))\\
&= \psi(A_{kl}b) \int_{G_n}{\cf_{AK_{n}^{(m)}}\ d\mu_{\Theta}}.
\end{align*}
Since $\psi(\tr(A_{kl}b))\neq 1$ the claim follows.
\end{proof}

\begin{Cor}\label{vanish2}
Let $\chi\colon F^{\ast}\to\C^{\ast}$ be a character of conductor $\mathfrak{f}(\chi)=\mathfrak{p}^{m}$ with $m\geq 0$ and let $A\in G_{n}$ with $\ord(A)<-\max(m,1)$.
Then we have
\begin{align*}
\int_{G_{n}}{(\chi\circ\det) \cdot \cf_{A K_{n}}\ d\mu_{\Theta}}=0.
\end{align*}
\end{Cor}
\begin{proof}
Let $m'=\max(m,1)$. We can rewrite the integral as
\begin{align*}
\int_{G_{n}}{(\chi\circ\det)\cdot\cf_{A K_{n}}\ d\mu_{\Theta}}
&=\sum_{k\in K_{n}/K_{n}^{(m')}} \chi(\det(Ak)) \int_{G_{n}}{ \cf_{AkK_{n}^{(m')}} d\mu_{\Theta}}.
\end{align*}
Using the fact that $\ord(Ak)=\ord(A)$ for all $k \in K_{n}$ the claim follows from Lemma \ref{vanish}.
\end{proof}

\begin{Lem}\label{gauss}
Let $\chi\colon F^{\ast}\to\C^{\ast}$ be a character of conductor $\mathfrak{f}(\chi)=\mathfrak{p}^{m}$ with $m\geq 1$ and let $A\in G_{n}$ be a matrix with $\ord(A)>-m$.
We have the following equality:
\begin{align*}
\int_{G_{n}}{(\chi\circ\det) \cdot \cf_{A K_{n}}\ d\mu_{\Theta}}=0.
\end{align*}
\end{Lem}
\begin{proof}
Firstly, let us assume we are in the case $m\geq 2$.
We are going to prove the stronger statement:
\begin{align*}
\int_{G_{n}}{(\chi\circ\det) \cdot \cf_{A K_{n}^{(m-1)}}\ d\mu_{\Theta}}=0.
\end{align*}
We have $AB\in M_{n}(\mathcal{O})$ for every $B\in \varpi^{m-1}M_{n}(\mathcal{O})$.
Therefore, using Lemma \ref{addequi} and property \eqref{prop1} we have
\begin{align*}
\int_{G_{n}}{(\chi\circ\det) \cdot \cf_{A K_{n}^{((m-1)}}\ d\mu_{\Theta}}
&= \lambda \left( \delta_{\rho}((\chi\circ\det) \cdot \cf_{A K_{n}^{(m-1)}})\right) \\
&= \lambda \left( 
\begin{pmatrix}
	1_n & AB \\ 0 & 1_n
\end{pmatrix}
\delta_{\rho}((\chi\circ\det) \cdot \cf_{A K_{n}^{(m-1)}})\right) \\
&= \lambda \left( \delta_{\rho}\left(AB\star\left((\chi\circ\det) \cdot \cf_{A K_{n}^{(m-1)}}\right)\right)\right).
\end{align*}
Taking the average over all $B\in \varpi^{m-1}M_n(\mathcal{O})$ yields
\begin{align*}
&\int_{\varpi^{m-1}M_n(\mathcal{O})}\left(AB\star\left((\chi\circ\det) \cdot \cf_{A K_{n}^{(m-1)}}\right)\right)(Ak)\ dB \\
=& \int_{\varpi^{m-1}M_n(\mathcal{O})} \chi(\det(Ak +AB) )\ dB \\
=& \chi(\det(Ak)) \int_{\varpi^{m-1}M_n(\mathcal{O})} \chi(\det(1_n +k^{-1}B ))\ dB \\
=& \chi(\det(Ak)) \int_{\varpi^{m-1}M_n(\mathcal{O})} \chi(\det(1_n +B) )\ dB \\
=& \chi(\det(Ak)) \int_{K_{n}^{(m-1)}} \chi(\det(k^{\prime}) )\ dk^{\prime}
\end{align*}
for every $k\in K_{n}^{(m-1)}$. Since $\det\colon K_{n}^{(m-1)} \to U^{(m-1)}$ is surjective,
the character $\chi\circ\det\colon K_{n}^{(m-1)}\to \C^{\ast}$ is non-trivial.
Hence, by orthogonality of characters the last integral vanishes.

The case $m=1$ can be proven in the same manner using property \eqref{prop2}.
\end{proof}

\begin{Def}
Let $\chi\colon F^{\ast}\to\C^{\ast}$ be a quasicharacter of conductor $\mathfrak{f}(\chi)=\mathfrak{p}^{m}$ with $m\geq 0$ and $a\in F^{\ast}$ with $\nu(a)=-m$
 We define the Gauss sum of $\chi$ (with respect to $\psi$) as
\begin{align*}
\tau(\chi):=\tau(\chi,\psi):=[U:U^{(m)}]\int_{U}{\chi(ag)\psi(ag) d^{\ast}g}.
\end{align*}
\end{Def}
\noindent For $r=(r_1,\ldots, r_{n})\in \Z^{n}$ we let $T_{r}\in G_n$ be the diagonal matrix given by $(T_{r})_{ii}=\varpi^{r_{i}}$ for $1\leq i\leq n$.
\begin{Lem}\label{cond}
Let $\chi\colon F^{\ast}\to\C^{\ast}$ be a character of conductor $\mathfrak{f}(\chi)=\mathfrak{p}^{m}$ with $m\geq 1$ and $r=(r_{1},\ldots,r_{n})\in \Z^{n}$.
If $r_i=-m$ for all $1\leq i\leq n$, we have
\begin{align*}
\int_{G_n}{\chi \cdot \cf_{T_r I_{n}^{(m)}}\ d\mu_{\Theta}}=
\tau(\chi)^{n} \left(\alpha\ q^{(n-n^{2})/{2}}\right)^{-m} q^{\frac{n^{2}+n}{2}} \lambda(\vartheta(\partial_{\rho}(\cf_{\widetilde{I}_{n}^{(1)}})))
\end{align*}
and otherwise we have
$$\int_{G_n}{\chi \cdot \cf_{T_r I_{n}^{(m)}}\ d\mu_{\Theta}}=0.$$
\end{Lem}
\begin{proof}
In the following we will identify elements $\epsilon=(\epsilon_{1},\ldots,\epsilon_{n})\in (U/U^{(m)})^{n}$ with the corresponding diagonal matrices in $\GL_{n}(\mathcal{O}/\p^{m})$ (resp.~with representatives in $K_n$). We have
\begin{align*}
\int_{G_n}{\chi \cdot \cf_{T_r I_{n}^{(m)}}\ d\mu_{\Theta}}
=& \sum_{\epsilon\in (U/U^{(m)})^{n}} \int_{G_n}{\chi \cdot \cf_{T_r \epsilon\widetilde{I}_{n}^{(m)}}\ d\mu_{\Theta}}\\
=& \sum_{\epsilon\in (U/U^{(m)})^{n}} \prod_{i=1}^{n}\chi(\varpi^{r_i}\epsilon_i){\int_{G_n} \cf_{T_r \epsilon\widetilde{I}_{n}^{(m)}}\ d\mu_{\Theta}}.
\end{align*}
Applying \eqref{addequi} with the matrix $T_r(\epsilon-1_n)$ yields
\begin{align*}
\int_{G_n}{\chi \cdot \cf_{T_r I_{n}^{(m)}}\ d\mu_{\Theta}}
=&\hspace{-0,7em} \sum_{\epsilon\in (U/U^{(m)})^{n}} \prod_{i=1}^{n}\chi(\varpi^{r_i}\epsilon_i)\psi(\varpi^{r_i}(\epsilon_i-1))\int_{G_n}{ \cf_{T_r \widetilde{I}_{n}^{(m)}}\ d\mu_{\Theta}}\\
=& \prod_{i=1}^{n}\psi(-\varpi^{r_i})\hspace{-0,7em} \sum_{\varepsilon_i\in U/U^{(m)}} \hspace{-1em}\chi(\varpi^{r_i}\varepsilon_{i})\psi(\varpi^{r_i}\varepsilon_{i}) \int_{G_n}{ \cf_{T_r \widetilde{I}_{n}^{(m)}}\ d\mu_{\Theta}}.
\end{align*}
Lemma \ref{gauss} in the case $n=1$ implies that the sum $\sum_{\varepsilon_i\in U/U^{(m)}}\chi(\varpi^{r_i}\varepsilon_{i})\psi(\varpi^{r_i}\varepsilon_{i})$ vanishes unless $r_{i}=-m$ for all $i$.

So let us assume for the rest of the proof that $r_{i}=-m$ for all $i$.
By the definition of the Gauss sum we get
\begin{align*}
\int_{G_n}{\chi \cdot \cf_{T_r I_{n}^{(m)}}\ d\mu_{\Theta}}
&= \tau(\chi)^{n} \psi(-\varpi^{-m})^{n}  \int_{G_n}{ \cf_{T_r \widetilde{I}_{n}^{(m)}}\ d\mu_{\Theta}}.
\end{align*}
The invariance property of Shalika functionals implies that the distribution
$$\lambda\circ\vartheta\circ\partial_{\rho^{\prime}}\colon C_{c}(M_{n}(F),\C)\too \C$$
is a multiple of the distribution $\psi\cdot dX$ for every $\rho^{\prime}\in \pi$.
Thus, the transformation law of $dX$ under linear transformations gives
\begin{align*}\psi(1-\varpi^{-m})^{n}  \int_{G_n}{ \cf_{T_r \widetilde{I}_{n}^{(m)}}\ d\mu_{\Theta}}
&=\psi(\varpi^{-m})^{n}  \lambda(\vartheta(\delta_{\rho}(\cf_{T_r \widetilde{I}_{n}^{(m)}})))\\
&=\psi(\varpi^{-m})^{n}  \lambda(\vartheta(\partial_{(1,T_{r})\rho}(\cf_{T_r \widetilde{I}_{n}^{(m)}})))\\
&= \alpha^{-m} \psi(\varpi^{-m})^{n} \lambda(\vartheta(\partial_{\rho}(\cf_{T_r \widetilde{I}_{n}^{(m)}})))\\
&= \alpha^{-m} q^{mn^{2}}\lambda(\vartheta(\partial_{\rho}(\cf_{\widetilde{I}_{n}^{(m)}})))\\
&= \alpha^{-m} q^{mn^{2}}[\widetilde{I}_{n}^{(1)}\colon\widetilde{I}_{n}^{(m)}]^{-1}\lambda(\vartheta(\partial_{\rho}(\cf_{\widetilde{I}_{n}^{(1)}})))
\end{align*}
and therefore, we can conclude that
\begin{align*}
\int_{G_n}{\chi \cdot \cf_{T_r I_{n}^{(m)}}\ d\mu_{\Theta}}
&= \alpha^{-m} q^{mn^{2}}[\widetilde{I}_{n}^{(1)}\colon\widetilde{I}_{n}^{(m)}]^{-1}\lambda(\vartheta(\partial_{\rho}(\cf_{\widetilde{I}_{n}^{(1)}})))\\
&= \tau(\chi)^{n} \left(\alpha\ q^{(n-n^{2})/{2}}\right)^{-m} q^{\frac{n^{2}+n}{2}} \lambda(\vartheta(\partial_{\rho}(\cf_{\widetilde{I}_{n}^{(1)}}))).
\end{align*}
\end{proof}

Now let us get back to the situation of Proposition \ref{integrals}.
In particular, we have fixed the unramified stabilization $\Theta^{\ur}=(\pi^{\ur},\rho^{\ur},\vartheta^{\ur})$ associated to the unramified irreducible principal series representation $V=\Ind^{G_{2n}}_{B_{2n}}(\chi_1,\ldots, \chi_{2n})$ with $\chi_{i}=\eta\chi_{2n-i+1}^{-1}$ for all $1\leq i\leq n$.
Hence we know that $$\alpha=\alpha_{\Theta^{\ur}}=\prod_{i=n+1}^{2n}\chi_i(\varpi).$$ As before, we write $\beta_{i}=\chi_{i}(\varpi)q^{n-i+1/2}=\alpha_{i}q^{n-i+1/2}$ for the Satake parameters of $V$.

\begin{Thm}\label{computation}
Let $\chi\colon F^{\ast}\to\C^{\ast}$ be a character of conductor $\mathfrak{f}(\chi)=\mathfrak{p}^{m}$. If the complex norm $\left|\chi(\varpi)\right|_{\infty}$ is sufficiently small, we have
\begin{align*}
E(\Theta^{\ur},\chi,1/2)=c\ \tau(\chi)^{n}\times\begin{cases}
\left(\alpha\ q^{(n-n^{2})/2}\right)^{-m} & \mbox{if}\ m\geq 1,\\
 \prod_{i=1}^{n}\frac{1-\beta_{n+i}^{-1}\chi(\varpi)^{-1}q^{-1/2}}{1-\beta_{n+i}\chi(\varpi)q^{-1/2}} & \mbox{if}\ m=0,\\
\end{cases}
\end{align*}
where $c$ is a non-zero rational constant independent of $\chi$.
\end{Thm}
\begin{proof}
We first treat the case $m \geq 1$. 
By Corollary \ref{vanish2} and Proposition \ref{integrals} we have 
\begin{align}\begin{split}\label{somestuff}
E(\Theta^{\ur},\chi,1/2)&= \int_{G_n} \chi\ d\mu_{\Theta^{\ur}}=\int_{\mathcal{F}^{m}} \chi\ d\mu_{\Theta^{\ur}}\\
&=\int_{\mathcal{F}^{m}} \rho(X) \chi(\det (X)) \psi^{-1}(\tr (X))\ dX\\
&=\hspace{-0,5em}\sum_{A\in \mathcal{F}^{m}/K_n} \int_{A K_n} \rho(X) \chi(X) \psi^{-1}(\tr(X))\ dX
, \end{split}
\end{align}
where $\mathcal{F}^{m}\subset G_n$ is the set of matrices $A$ with $\ord(A)\geq -m$.
By the Iwasawa decomposition every coset $AK_n\in \mathcal{F}^{m}/K_n$ has a representative of the form $T_{r(AK_n)} +N$ with $r(A)\in \Z^{n}$ and $N$ a nilpotent upper triangular matrix with entries in $\p^{-m}$.
The diagonal matrix $T_{r(AK_n)}$ is uniquely determined by the coset $AK_n$. For a fixed $r\in \Z^{n}$ we define
$$\mathcal{F}_{r}^{m}=\left\{A\in\mathcal{F}^{m}\mid r(AK_n)=r \right\}.$$
A complete set of representatives of $\mathcal{F}_{r}^{m}/K_n$ is given by $\left\{T_{r}+N\right\}$, where $N$ are nilpotent upper triangular matrices with entries $N_{ij}$ running through a set of representatives of $\mathfrak{p}^{-m}/\mathfrak{p}^{r_{i}}$ for all $j>i$.
If we let the $N_{ij}$ run through a set of representatives of $\mathfrak{p}^{-m}/\mathfrak{p}^{r_{i}+m}$ instead, we get exactly $q^{m(n^{2}-n)/2}$ representatives for each equivalence class in $\mathcal{F}_{r}^{m}$.
Therefore, we get
\begin{align*}
&\sum_{A\in\mathcal{F}_{r}^{m}/K_n}\int_{K_{n}}{\chi(\det(Ag))\psi(\tr(Ag))d^{\ast}g}\\
=&\frac{1}{q^{m(n^{2}-n)/2}}\sum_{N}\int_{K_{n}}{\chi(\det(T_{r}g))\psi(\tr(T_{r}g))\psi(\tr(Ng))d^{\ast}g}\\
=&\frac{1}{q^{m(n^{2}-n)/2}}\int_{K_{n}}{\chi(\det(T_{r}g))\psi(\tr(T_{r}g))\sum_{N}\psi(\tr(Ng))d^{\ast}g}.
\end{align*}
By orthogonality of characters the sum
\begin{align*}
\sum_{N}\psi(\tr(Ng))&=\sum_{N_{ij}\in\mathfrak{p}^{-m}/\mathfrak{p}^{r_{i}+m}} \prod_{j>i} \psi(N_{ij}g_{ji})\\
&=\prod_{j>i} \sum_{N_{ij}\in\mathfrak{p}^{-m}/\mathfrak{p}^{r_{i}+m}} \psi(N_{ij}g_{ji})
\end{align*}
is zero unless $\nu(g_{ji})\geq m$ for all $j>i$.
Therefore, using the proof of Lemma \ref{cond} most of the terms in \eqref{somestuff} vanish and we are left with
\begin{align*}
E(\Theta^{\ur},\chi,1/2)= \int_{G_n}{\chi \cdot \cf_{T_{(-m,\ldots,-m)}} I_{n}^{(m)}\ d\mu_{\Theta}}.
\end{align*}
The claim now follows by using Lemma \ref{cond} once again.

\noindent Now let $\chi$ be unramified. Corollary \ref{vanish2} and Proposition \ref{integrals} give
\begin{align*}
E(\Theta^{\ur},\chi,1/2)&=\int_{\mathcal{F}^{1}}\chi\ d\mu_{\Theta^{\ur}}=\int_\mathcal{F}^{1} \rho(X) \chi(\det (X)) \psi^{-1}(\tr (X))\ dX
\end{align*}
Invoking the Iwasawa decomposition and the comparison between multiplicative and additive Haar measures we get
\begin{align*}
E(\Theta^{\ur},\chi,1/2)=\lim_{k\to\infty} \sum_{\substack{(r_{i})\in\Z^{n}\\ -1\leq r_{i}\leq k}} \prod_{i=1}^{n}q^{-r_{i}n}\alpha_{n+i}^{r_{i}}\chi(\varpi^{r_{i}})\sum_{A \in\mathcal{F}^{1}_{r}} \int_{K_{n}}{\psi(\tr(Ag))\ d^{\ast}g}.
\end{align*}
As in the case $m\geq 1$ we see that
\begin{align*}
\sum_{A \in\mathcal{F}^{1}_{r}} \int_{K_{n}}{\psi(\tr(Ag))\ d^{\ast}g}=\prod_{i=1}^{n}q^{(1+r_{i})(n-i)} \int_{I^{(1)}_n}{\psi (\tr(T_{r}g))\ d^{\ast}g}.
\end{align*}
The integral on the right hand side of the equality is equal to $\vol(I^{(1)}_n)$ if $r_{i}\geq 0$ for all $1\leq i\leq n$.
Otherwise a computation with diagonal matrices like before shows that
\begin{align*}
\int_{I^{(1)}_n}{\psi (\tr(T_{r}g)\ d^{\ast}g}=\vol(I^{(1)}_n) \prod^{n}_{\substack{i=1\\r_{i}=-1}}\left(-\frac{1}{q-1}\right).
\end{align*}
Therefore, we get that
\begin{align*}
E(\Theta^{\ur},\chi,1/2)&=\vol(I^{(1)}_n)\prod_{i=1}^{n}{-\frac{1}{q-1}\alpha_{n+i}^{-1}\chi(\varpi)^{-1}q^n+q^{n-i}\sum_{k=0}^{\infty}{(\alpha_{n+i}\chi(\varpi)q^{-i})^{k}}}\\
&=\vol(I^{(1)}_n)\prod_{i=1}^{n}{-\frac{1}{q-1}\alpha_{n+i}^{-1}\chi(\varpi)^{-1}q^n +q^{n-i}\frac{1}{1-\alpha_{n+i}\chi(\varpi)q^{-i}}}\\
&=\vol(I^{(1)}_n)\prod_{i=1}^{n}\frac{q^{n-i+1}}{q-1}\frac{1-\alpha_{n+i}^{-1}\chi(\varpi)^{-1}q^{i-1}}{1-\alpha_{n+i}\chi(\varpi)q^{-i}}\\
&=\vol(I^{(1)}_n)[U:U^{(1)}]^{-n}q^{\frac{n^{2}+n}{2}} \prod_{i=1}^{n}\frac{1-\beta_{n+i}^{-1}\chi(\varpi)^{-1}q^{-1/2}}{1-\beta_{n+i}\chi(\varpi)q^{-1/2}}.
\end{align*}
The constants appearing in the all cases are equal, which proves the claim.
\end{proof}

\end{subsection}
\begin{subsection}{The semi-local case}\label{semilocal}
All the previous constructions can be easily generalized to the semi-local case.
We change our notations slightly.
Let $F_{1},\ldots,F_{g}$ be finite extensions of $\Q_{p}$. We put $F=F_{1}\times\ldots\times F_{g}$ and $\mathcal{O}=\mathcal{O}_{F_{1}}\times\ldots\times\mathcal{O}_{F_{g}}$.
For every $r\in \N$ we define $G_{r}=\GL_{r}(F)$ and $K_{r}=\GL_{r}(\mathcal{O})$.

Further, we fix a subfield $E\subset \C_p$ which is a finite extension of $\Q_{p}$ such that every embedding of every $F_i$ into $\C_p$ factors through $E$ for all $1\leq i \leq g$ (or, an arbitrary field $E$ of characteristic $0$ in the smooth case).
\begin{Def}\label{semistab}
Let $V_i$ be irreducible locally $\Q_p$-rational $\GL_{2n}(F_{i})$-representations on $E$-vector spaces for every $1\leq i \leq g$.
\begin{enumerate}[(i)]
\item A stabilization $\Theta$ of the $G_{2n}$-representation $$V=\bigotimes_{i=1}^{g} V_i$$ is a tuple $\Theta=(\Theta_i)_{1\leq i \leq g}$, where $\Theta_i$ is a stabilization of $V_i$ for $1\leq i \leq g$.
\item A stabilization $\Theta$ of $V$ is weakly ordinary if each of the $\Theta_i$ is weakly ordinary.
\item A critical point $V_s$ of $V$ is a tensor product of the form
\begin{align*}V_s=\bigotimes_{i=1}^{g}V_{s,i},\end{align*}
where $V_{s,i}$ is a critical point of $V_i$ for all $1\leq i\leq g$.
\end{enumerate}
\end{Def}
\noindent The map given by
\begin{align*}
\bigotimes_{i=1}^{r}C_c^{0}(\GL_{n}(F_{k}),E)&\too C_c^{0}(G_n,E),\\
f_1\otimes\ldots\otimes f_g &\mapstoo [(g_1,\ldots,g_n)\mapsto f_1(g_1)\cdot\ldots\cdot f_n(g_n)] 
\end{align*}
gives an isomorphism of $G_n\times G_n$-representations.
Therefore, every data of a stabilization $\Theta$ and a critical point $V_s$ of a representation $V$ gives rise to a map
\begin{align}\label{semidelta}
\delta_{\Theta,s}\colon C_c^{0}(G_n,E)\otimes V_{s}\mapstoo V.
\end{align} 
If $V_i$ is smooth for all $1 \leq i \leq g$, then $V_{s}$ is automatically trivial.
All the local results from Section \ref{int} carry over verbatim to the semi-local case.

Suppose we are in the smooth case. Let $K_{\Theta}=K_{\Theta,1}\times\cdots\times K_{\Theta,g}\subset K_{n}$ be the maximal open subgroup such that $\delta_{\Theta}$ is $1\times K_{\Theta}$-equivariant.
For an arbitrary ring $R$, which contains the central character $\omega$ of $V_{\sm}$, and an $R$-module $N$ we define
\begin{align*}
IC_{\omega}^{0}(K_{n},R)=\cind_{Z(K_{n}\times K_{\Theta})}^{G_n\times G_{n}} \omega\otimes C_c^{0}(K_n,R)
\end{align*}
and
\begin{align*}
\iDist(K_{n},N)=\cind_{Z(K_{n}\times K_{\Theta})}^{G_n\times G_{n}} \omega^{-1}\otimes \Dist(K_n,N).
\end{align*}
where $Z$ is center of $G_{2n}$, which we view as a subgroup of $G_n\times G_{n}$ via the diagonal embedding.
By Frobenius reciprocity the map $\delta_{\Theta}$ induces a map
\begin{align}\label{dualdos}
IC_{\omega}^{0}(K_{n},E)\mapstoo V,
\end{align}
which we also denote by $\delta_{\Theta}$.

\end{subsection}

\end{section}
\color[rgb]{0,0,0}
\begin{section}{The global distribution}\label{global}
We will use the following notations throughout the rest of the article.
We fix a totally real algebraic number field $F$ of degree $d$ with ring of integers $\mathcal{O}$. For a non-zero ideal $\mathfrak{a}\subset\mathcal{O}$ we set $N(\mathfrak{a})=\sharp (\mathcal{O}/\mathfrak{a})$.
Given a place $l$ of $\Q$ we denote by $S_{l}$ the set of places of $F$ above $l$.
Let $\sigma_{1},\ldots,\sigma_{d}$ denote the distinct embeddings of $F$ into $\R$ and $\infty_{1},\ldots,\infty_{d}$ the corresponding Archimedean places.
Via the fixed embedding $\iota_{\infty}\colon\overline{\Q}\to\C$ we can and will view the $\sigma_{i}$ as embeddings into $\overline{\Q}$.

If $v$ is a place of $F$, we denote by $F_{v}$ the completion of $F$ at $v$.
If $\mathfrak{q}$ is a finite place, we let $\mathcal{O}_{\mathfrak{q}}$ denote the valuation ring of $F_{\mathfrak{q}}$ and $\ord_{\mathfrak{q}}$ the additive valuation such that $\ord_{\mathfrak{q}}(\varpi)=1$ for any local uniformizer $\varpi\in\mathcal{O}_{\mathfrak{q}}$.
For an arbitrary place let $|\cdot|_{v}$ be the normalized multiplicative norm, i.e. $|\cdot|_{\infty_{i}}=|\sigma_{i}(\cdot)|$ for $i=1,\ldots,d$ and $|\cdot|_{\mathfrak{q}}=N(\mathfrak{q})^{-\ord_{q}(\cdot)}$ if $\mathfrak{q}$ is a finite place.
We denote by $U_{v}$ the invertible elements of $\mathcal{O}_{v}$ if $v$ is a finite place and the group of positive elements of $F_{v}$ if $v$ is a real place.
For a finite place $\mathfrak{q}$ we let $U^{(m)}_{\mathfrak{p}}=\left\{x\in U_{\mathfrak{q}}\left|x\equiv 1 \mod \mathfrak{q}^{m}\right.\right\}$.

Let $\A$ be the ring of adeles of $F$ and $\I$ the idele group of $F$.
We denote by $|\cdot|\colon\I\to\R^{\ast}$ the absolute modules, i.e.~$|(x_{v})_{v}|=\prod_{v}{|x_{v}|_{v}}$ for $(x_{v})_{v}\in\I$.
For a finite set $S$ of places of $F$ we define the "$S$-truncated adeles" $\A^{S}$ (resp.~"$S$-truncated ideles" $\I^{S}$) as the restricted product of all completions $F_{v}$ (resp.~$F_{v}^{\ast}$) with $v\notin S$ and put $F_{S}=\prod_{v\in S}{F_{v}}$.
We also set $U_{S}=\prod_{v\in S}U_{v}$ and $U^{S}=\prod_{v\notin S}U_{v}$ and similarly we define $U^{(m)}_{S}$.
If $I$ is a finite set of places of $\Q$, we often write $\A^{I}$ instead of $\A^{\cup_{v\in I}S_{l}}$, $U_{I}$ instead of $\prod_{v\in I}U_{S_{l}}$ etc.

A (left) Haar measure of a locally compact group $G$ will be denoted by $dg$.
We fix a non-trivial character $\psi\colon\A\to S^{1}$ which is trivial on $F$.
For a place $v$ let $\psi_{v}$ be the restriction of $\psi$ to $F_{v}\subset\A$.
We assume that the conductor of $\psi_{\mathfrak{p}}$ is $\mathcal{O}_{\mathfrak{p}}$ for all places $\mathfrak{p}\in S_p$.
Let $dx$ (resp.~$dx_{v}$) denote the self-dual Haar measure of $M_{r}(\A)$ (resp.~$M_{r}(F_{v})$) associated to the character $\psi\circ\tr$ (resp.~$\psi_{v}\circ\tr$).
It follows that $dx=\prod_{v}dx_{v}$. We normalize the multiplicative Haar measure $d^{\ast}x_{v}$ on $\GL_{r}(F_{v})$ by $d^{\ast}x_{v}=m_{v}\frac{dx_{v}}{|x_{v}|_{v}}$, where $m_{v}=1$ if $v$ is real and $m_{v}$ is chosen such that $\GL_r(O_{v})$ has volume $1$ if $v$ is finite.
 For a linear algebraic group $\mathbb{G}$ over $F$, a character $\chi\colon \mathbb{G}(\A)\to\C^{\ast}$ and a place $v$ we let $\chi_{v}\colon \mathbb{G}(F_{v})\hookrightarrow \mathbb{G}(\A)\stackrel{\chi}{\longrightarrow}\C^{\ast}$ be the local component of $\chi$ at $v$.
Further, we write $P\mathbb{G}$ for the quotient of $\mathbb{G}$ by its center.
Given an algebraic character $\dc\colon \mathbb{G}\to \mathbb{G}_{m,F}$ we denote by $\mathbb{G}(F_{\infty})^{+}\subset \mathbb{G}(F_{\infty})$ the subgroup of elements which have totally positive image under $\dc$.
Similarly, we define $\mathbb{G}(F)^{+}$ etc. In the case $\mathbb{G}\subset\mathbb{GL}_{r}$ the superscript $+$ is always meant with respect to the determinant.
Finally, we write $\mathcal{G}_{p}$ for the Galois group of the maximal abelian extension of $F$ unramified outside $p$ and $\infty$.
\begin{subsection}{Shalika models}
In this section we recall the basics on global Shalika models and their connection to $L$-functions. The main reference is \cite{FJ}.
Until the end of the article we denote by $G$ the algebraic group $\mathbb{GL}_{2n}$ and by $Z$ its center. We write $B$ for be the Borel subgroup of upper triangular matrices in $G$.
We view $H=\mathbb{GL}_{n}\times\mathbb{GL}_{n}$ as an algebraic subgroup of $G$ via the diagonal embedding. For $m_1, m_2\in \Z$ we define the morphism of algebraic groups
$$\det^{m_{1},m_{2}}\colon H\too\mathbb{G}_{m},\ (g_{1},g_{2})\mapstoo\det(g_{1})^{m_{1}}\det(g_{2})^{m_{2}}.$$
The Shalika subgroup $S$ of $G$ is defined as
\begin{align*}
S=\left\{\left.\begin{pmatrix}h&0\\0&h\end{pmatrix}\begin{pmatrix}1_n&X\\0&1_n\end{pmatrix}\right| h\in GL_{n}, X\in M_{n}\right\}.
\end{align*}
For the rest of this article we fix a continuous character $\eta\colon \I/F^{\ast}\to \C^{\ast}$.
It induces a character $\eta\psi\colon S(\A)\to \C^{\ast}$ via
\begin{align*}
\begin{pmatrix}h&0\\0&h\end{pmatrix}\begin{pmatrix}1_n&X\\0&1_n\end{pmatrix}\mapsto \eta(\det(h))\psi(\tr(X)).
\end{align*}
Let $V=\otimes_{v} V_{v}$ be a cuspidal automorphic representation of $G(\A)$ with central character $\omega=\eta^{n}$.
\begin{Def}
The cuspidal representation $V$ has a (\textit{global}) $(\eta,\psi)$-\textit{Shalika model} if there exist $\Phi\in V$ and $g\in G(\A)$ such that the following integral does not vanish: 
\begin{align*}
\Xi_{\Phi}(g)=\int_{Z(\A)S(F)\backslash S(\A)}{(\pi(g)\Phi)(s)(\eta\psi(s))^{-1}ds}.
\end{align*}
\end{Def}
The integral is well-defined since $\Phi$ is a cusp form.
The global Shalika functional $\Lambda\colon V\to \C$ is defined by $\Lambda(\Phi)=\Xi_{\Phi}(1)$. By a simple change of variables we see that $\Lambda(s\Phi)=\eta\psi(s)\Lambda(\Phi)$ holds for all $s\in S(\A)$ and $\Phi\in V$.
We will assume from now on that $V$ has a $(\eta,\psi)$-Shalika model.

\begin{Exa}
\begin{enumerate}[(i)]\thmenumhspace
\item If $n=1$, a Shalika functional is the same as a Whittaker functional. Thus, every cuspidal automorphic representation of $GL_{2}(\A)$ has a Shalika model.
\item Let $V$ be a cuspidal representation of $\GL_{2}(\A)$, which is neither of dihedral nor of tetrahedral type.
Then, by the work of Kim and Shahidi (cf.~\cite{KS}) the symmetric cube lift $\Pi=\Sym^{3}(V)$ is cuspidal.
It is well-known that in this case $\Pi$ has a Shalika model (see for example \cite{GRG}, Proposition 8.1.1). 
\end{enumerate}
\end{Exa}

\begin{Pro}\label{globzeta}
Let $f\colon \I/F^{\ast}\to\C$ be a locally constant function, $\Phi$ an element of $V$ and $s\in\C$.
Then the integral
\begin{align*}
\Psi(\Phi,f,s)=\int_{Z(\A)H(F)\backslash H(\A)}\hspace{-1,2em}{\Phi(h)\left|\det^{1,-1}(h)\right|^{s-1/2}f\left(\det^{1,-1}(h)\right)\eta^{-1}(\det^{0,1}(h))dh}
\end{align*}
converges absolutely and defines an holomorphic function in $s$.
If $\mathfrak{Re}(s)$ is sufficiently large, it equals the following absolutely convergent integral:
\begin{align*}
Z(\Phi,f,s)=\int_{GL_{n}(\A)}{\Xi_{\Phi}\left(\begin{pmatrix}g&0\\0&1_n\end{pmatrix}\right)}f(\det(g))\left|\det(g)\right|^{s-1/2}d^{\ast}g.
\end{align*}
\end{Pro}
\begin{proof}
If $f=\chi\colon\I/F^{\ast}\to\C^{\ast}$ is a character and $V$ is unitary, this is precisely Proposition 2.3 of \cite{FJ}.
By twisting with a character we may assume that $V$ is unitary. Since the locally constant characters form a basis of $C^{0}(\I/F^{\ast},\C)$ the claim follows.
\end{proof}

\noindent The global Shalika functional factors as a product of local Shalika functionals in the following sense:
There exist non-zero functionals $\lambda_{v} \colon V_{v} \to \C$ for every place $v$ of $F$ such that $\Lambda(\Phi)=\prod_{v}{\lambda_{v}(\varphi_{v})}$ holds for all pure tensors $\Phi=\otimes_{v}\varphi_{v}\in V=\otimes V_{v}$.
Hence, we have the equality $\Xi_{\Phi}=\prod_{v}{\xi_{\varphi_{v}}^{\lambda_{v}}}$ (see \ref{density} for the definition of ${\xi_{\varphi_{v}}^{\lambda_{v}}}$).
Moreover, we have $\lambda_{v}(s_{v}\varphi_{v})=\eta_{v}\psi_{v}(s_{v})\lambda_{v}(\varphi_{v})$ for all $\varphi\in V_{v}$ and $s_{v}\in S(F_{v})$.
Thus, $\lambda_{v}$ is a local Shalika functional of $V_{v}$ as in Definition \ref{shalika} for every finite place $v$.

\begin{Pro}\label{loczeta} Let $v$ be a finite place, $\varphi_{v}\in V_{v}$ and $\chi_{v}\colon F_{v}^{\ast}\to\C^{\ast}$ a character.
Then for every complex number $s\in\C$ with sufficiently large real part the following local zeta integral converges absolutely:
\begin{align*}
\zeta_{v}(\varphi_{v},\chi_{v},s)=\int_{GL_{n}(F_{v})}{\xi^{\lambda_{v}}_{\varphi_{v}}\left(\begin{pmatrix}g&0\\0&1_n\end{pmatrix}\right)}\chi_{v}(\det(g))\left|\det(g)\right|_{v}^{s-1/2}d^{\ast}g.
\end{align*} 
There exists $\varphi_{v}\in V_{v}$ such that $L(V_{v}\otimes\chi_{v},s)=\zeta_{v}(\varphi_{v},\chi_{v},s)$ holds for $\mathfrak{Re}(s)$ large enough and all unramified characters $\chi_{v}\colon F^{\ast}\to \C^{\ast}$.
Moreover, if $V_{v}$ is unramified, this equality holds for a spherical vector.
By our choice of Haar measure on $\GL_{2n}(F_{v})$ we can choose the normalized spherical vector for almost all $v$.
\end{Pro}
\begin{proof}
See Proposition 3.1 and Proposition 3.2 of \cite{FJ} for the unitary case.
The non-unitary case follows by twisting with an appropriate character.
\end{proof}

\end{subsection}
\begin{subsection}{The global distribution}\label{globaldist}
The goal of this section is to construct the global distribution and show that it fulfills the right interpolation property in the nearly spherical case.

Let $V=\otimes_v V_{v}$ be a cuspidal automorphic representation of $G(\A)$ having a $(\eta,\psi)$-Shalika model.
Further, we assume that we have given a stabilization $\Theta$ of the semi-local representation $V_p=\otimes_{\p\in S_p}V_{\p}$.
Finally, we fix a finite set $\Sigma$ of finite places of $F$, which is disjoint from $S_{p}$.

For every integer $m\geq 1$ we define $\Phi^{\infty}_{m,\Sigma}=\otimes_{v\nmid\infty}\varphi_{m,\Sigma,v}\in\otimes_{v\nmid\infty}V_{v}$ to be the following pure tensor: 
\begin{itemize}
\item Case $v\notin S_p\cup \Sigma$: $\varphi_{v}=\varphi_{m,\Sigma,v}$ is chosen as in the end of Proposition \ref{loczeta}.
Especially, it is independent of $m$.
\item Case $v\in\Sigma$: $\varphi_{v,\Sigma}=\varphi_{m,\Sigma,v}$ is chosen such that
$$\zeta_{v}(\varphi_{v,\Sigma},\chi_{v},s)=1$$
holds for all $s$ and all unramified characters $\chi_{v}$.
See Section 3.9.3 of \cite{GRG} for an explicit construction of such a vector.
\item Case $\mathfrak{p}\in S_{p}$:
we define $\varphi_{m,\mathfrak{p}}=\varphi_{m,\Sigma,\mathfrak{p}}=\left[GL_{n}(\mathcal{O}_{\mathfrak{p}}):K^{(m)}_{n,\mathfrak{p}}\right]\cdot\delta_{\Theta_\p}(\cf_{K^{(m)}_{n,\mathfrak{p}}})$, where $K^{(m)}_{n,\mathfrak{p}}$ is the $m$-th principle congruence subgroup of $GL_{n}(\mathcal{O}_{\mathfrak{p}})$.
\end{itemize}
The choice of a vector $\Phi_{\infty}$ at infinity will be discussed at the end of section \ref{cohom}.
For now, we fix an arbitrary vector $\Phi_{\infty}=\otimes\varphi_{v}\in \otimes_{v|\infty} V_{v}$ and put $\Phi_{m,\Sigma}=\Phi^{\infty}_{m,\Sigma}\otimes \Phi_{\infty}$. 
If $f\colon\I/F^{\ast}\to\C$ is a locally constant function, there exists some integer $m\geq 1$ such that $f$ factors through $\I/F^{\ast}U^{(m)}_{S_{p}}$.
For every $s\in\C$ the integral
$$
\int_{\I/ F{\ast}}{f(x)|x|^{s}\mu_{\Theta_{\Sigma}}(dx)}:=\Psi(\Phi_{m,\Sigma},f,s+1/2)
$$
converges absolutely by Proposition \ref{globzeta} and defines a holomorphic function in $s$.
It is easy to see that the integral is independent of the choice of $m$.

By class field theory the Artin map $\rec\colon\I/F^{\ast}\to\mathcal{G}_{p}$ is continuous and surjective.
Hence, for every $s\in\C$ we can define a distribution $\mu_{\Theta_{\Sigma},s}\in \Dist(\mathcal{G}_{p},\C)$ by
$$
\int_{\mathcal{G}_{p}}{f(\gamma)\mu_{\Theta_{\Sigma},s}(d\gamma)}=\int_{\I/ F{\ast}}{f(\rec(x))|x|^{s}\mu_{\Theta_{\Sigma}}(dx)}
$$
for all $f\in C^{0}(\mathcal{G}_{p},\C)$.
In the following, we always identify a character on the Galois group $\mathcal{G}_{\p}$ with the corresponding idele class character.
\begin{Pro}[Interpolation property]\label{interpo}
For every character $\chi\colon\mathcal{G}_{p}\to\C^{\ast}$ we have (up to a non-zero scalar) the following equality:
\begin{align*}
\int_{\mathcal{G}_{p}}{\chi(\gamma)\mu_{\Theta_{\Sigma},s}(d\gamma)}=&\prod_{\mathfrak{p}\in S_{p}}{e(\Theta_{p},\chi_{\mathfrak{p}},s+1/2)}\times L_{S_{\infty}\cup\Sigma}(\pi\otimes\chi,s+1/2)\\
&\times \prod_{v\in S_{\infty}}\zeta_{v}(\varphi_{v},\chi_{v},s+1/2).
\end{align*}
\end{Pro}
\begin{proof}
Since both sides of the equation are holomorphic in $s$ it is enough to show that the equality holds for $\mathfrak{Re}(s)$ large.
For $m$ large enough we get
\begin{align*}
&\ \ \int_{\mathcal{G}_{p}}{\chi(\gamma)\mu_{\Theta_{\Sigma},s}(d\gamma)}\\
=&\ \ \Psi(\Phi_{m,\Sigma},\chi,s+1/2)\\
=&\ \ Z(\Phi_{m,\sigma},\chi,s+1/2)\\
=&\ \ \prod_{v}{\zeta_{v}(\varphi_{m,\Sigma,v},\chi_{v},s+1/2)}\\
=&\hspace{-0.65em}\prod_{v\notin S_{p,\infty}\cup\Sigma}\hspace{-1.4em}L(\pi_{v}\otimes\chi_{v},s+1/2)\prod_{\mathfrak{p}\in S_{p,\infty}}\hspace{-0.7em}{\zeta_{v}(\varphi_{m,\mathfrak{p}},\chi_{\mathfrak{p}},s+1/2)}
\end{align*}
by using Proposition \ref{loczeta} and by our choice of $\varphi_{m,\mathfrak{p}}$ for $\mathfrak{p}\notin S_{p,\infty}$.
Using Lemma \ref{density} we see that
\begin{align*}
\zeta_{\mathfrak{p}}(\varphi_{m,\mathfrak{p}},\chi_{\mathfrak{p}},s+1/2)=&\int_{GL_{n}(F_{\mathfrak{p}})}{\xi^{\lambda_{\mathfrak{p}}\circ\vartheta_p}_{\varphi_{m,\mathfrak{p}}}\left(\begin{pmatrix}g&0\\0&1_n\end{pmatrix}\right)}\chi_{\mathfrak{p}}(\det(g))\left|\det(g)\right|_{\mathfrak{p}}^{s}d^{\ast}g\\
&=\int_{GL_{n}(F_{\mathfrak{p}})}{\chi_{\mathfrak{p}}(\det(g))\left|\det(g)\right|_{\mathfrak{p}}^{s}\mu_{\vartheta_{\mathfrak{p}}}(dg)}\\
&=E(\Theta_{p},\chi_{\mathfrak{p}},s+1/2)\\
&=e(\Theta_{p},\chi_{\mathfrak{p}},s+1/2)\ L(V_{\p}\otimes\chi_{\p},s+1/2)
\end{align*}
holds for all $\mathfrak{p}\in S_p$.
\end{proof}
Now let us assume that for all $\p \in S_p$ the local representation $V_{\mathfrak{p}}$ is of the form
$$V_{\mathfrak{p}}=V^{\ur}_{\mathfrak{p}}\otimes\chi^{\prime}_{\mathfrak{p}},$$
where $V^{\ur}_{\mathfrak{p}}$ is a spherical representation and $\chi_{\mathfrak{p}}^{\prime}\colon F_{\p}^{\ast}\to \C^{\ast}$ is a character.
The representation $V^{\ur}_{\mathfrak{p}}$ is isomorphic to an unramified principal series representation for all $\p\in S_p$.
As in Section \ref{localshalika}, we can choose $\chi_{i,\p}\colon F_{\p}^{\ast}\to\C^{\ast}$ for all $1\leq i\leq 2n$ and all $\p\in S_p$ such that 
\begin{itemize}
	\item $V^{\ur}_\p=\Ind_{B(F_\p)}^{G(F_\p)}(\chi_1,\ldots,\chi_{2n})$ and
	\item $\chi_{i,\p}=\eta_\p\chi_{2n-i+1,\p}^{-1}$.
\end{itemize}
Let $\Theta_\p^{\ur}=(\pi^{\ur}_\p,\rho^{\ur}_\p,\vartheta^{\ur}_\p)$ the unramified stabilizations of $V^{\ur}_{\mathfrak{p}}$ associated to these data, i.e.~:
\begin{itemize}
	\item $\pi^{\ur}_\p= \Ind_{B_n(F_\p)}^{GL_n(F_\p)}(\chi_1,\ldots,\chi_n)\otimes \Ind_{B_n(F_\p)}^{GL_n(F_\p)}(\chi_{n+1},\ldots,\chi_{2n})$
	\item $\rho^{\ur}_\p$ is the unique normalized spherical vector in $\pi^{\ur}_\p$
	\item $\vartheta^{\ur}$ is the canonical isomorphism
\end{itemize}
The local Satake parameters of $V_\p$ are given by $\beta_{i,\p}=\chi_{i,\p}(\varpi)N(\p)^{n-i+1/2}$, where $\varpi$ is a local uniformizer at $\p$.
We set $\alpha_{\p}=\prod_{i=n+1}^{2n}\chi_{i,\p}(\varpi)$.
As in Section \ref{localshalika} we assume that the technical condition $\beta_{i,\p}\beta_{j,\p}\neq \eta^{\pm 1}(\varpi)$ holds for all $1\leq i<j\leq n$ and all $\p\in S_p$.

We consider the stabilization $\Theta^{\ur}\otimes\chi^{\prime}$ of $V_{p}$, whose local components are given by $\Theta_{\mathfrak{p}}^{\ur}\otimes\chi_{\mathfrak{p}}^{\prime}$.
For a character $\chi\colon\mathcal{G}_{p}\to\C^{\ast}$ we define
$$\mathfrak{f}(\chi^{\prime}\chi)=\prod_{\mathfrak{p}\in S_{p}} \mathfrak{f}(\chi^{\prime}_{\mathfrak{p}}\chi_{\mathfrak{p}})$$
and similarly
$$\tau(\chi^{\prime}\chi)=\prod_{\mathfrak{p}\in S_{p}} \tau(\chi^{\prime}_{\mathfrak{p}}\chi_{\mathfrak{p}},\psi_\mathfrak{p}^{-1}).$$
As an immediate consequence of Lemma \ref{dothetwist}, Proposition \ref{interpo} and Theorem \ref{computation} we get
\begin{Cor}[Interpolation property - the nearly spherical case]\label{interpol}
Under the assumptions above we have that for every character $\chi\colon\mathcal{G}_{p}\to\C^{\ast}$ the following equality holds (up to a non-zero scalar):
\begin{align*}
\int_{\mathcal{G}_{p}}{\chi(\gamma)\mu_{\Theta_{\Sigma}^{\ur}\otimes\chi^{\prime},s}(d\gamma)}=&N(\mathfrak{f}(\chi^{\prime}\chi))^{ns}\tau(\chi^{\prime}\chi)^{n} \prod_{\mathfrak{p}\in S_{p}}{e^{\prime}(\Theta^{\ur}_{\mathfrak{p}}\otimes\chi^{\prime}_{\mathfrak{p}},\chi_{\mathfrak{p}},s+1/2)}\\
&\times L_{S_{\infty}\cup\Sigma}(\pi\otimes\chi,s+1/2)\prod_{v\in S_{\infty}} \zeta_{v}(\varphi_{v},\chi_{v},s+1/2),
\end{align*}
where the modified Euler factor $e^{\prime}(\Theta^{\ur}_{\mathfrak{p}}\otimes\chi^{\prime}_{\mathfrak{p}},\chi_{\mathfrak{p}},s+1/2)$ is equal to
\begin{align*}
\begin{cases}
\prod_{i=1}^{n}(1-\beta_{i}\chi(\varpi)q^{-s-1/2})(1-\beta_{n+i}^{-1}\chi(\varpi)^{-1}q^{s-1/2}) &\mbox{if}\ \ord_{\mathfrak{p}}(\mathfrak{f}(\chi^{\prime}\chi))=0\\
(\mathcal{N}(\mathfrak{p})^{\frac{n^{2}-n}{2}}\alpha_{\mathfrak{p}})^{-\ord_{\mathfrak{p}}(\mathfrak{f}(\chi^{\prime}\chi))}&\mbox{if}\ \ord_{\mathfrak{p}}(\mathfrak{f}(\chi^{\prime}\chi))>0.
\end{cases}
\end{align*}
\end{Cor}
\end{subsection}
\end{section}

\begin{section}{Boundedness of the distribution}\label{last}
For a cohomological cuspidal representation and a critical half-integer $s+1/2$ we are going to recast the definition of the distribution in terms of group cohomology.
As an immediate consequence the rationality of the distribution follows.
Further, we show that the weak ordinarity condition combined with the existence of lattices, which are homologically of finite type, implies the boundedness of the distribution.
Let us fix a cuspidal automorphic representation $V$ of $G(\A)$ with central character $\omega\colon \I/F^{\ast}\to \C^{\ast}$. We put $V_p=\otimes_{\mathfrak{p}\in S_p}V_\mathfrak{p}$ and $V_\infty=\otimes_{v\in S_{\infty}}V_v$.
\begin{subsection}{Cohomology classes attached to characters}\label{characters}
Before attaching cohomology classes to automorphic forms we attend to the simpler question of how to give a cohomological description of distributions and characters.

Let $R$ be a ring. The Artin reciprocity map induces a surjective map $\I^{\infty}/U^{p,\infty}\to\mathcal{G}_{p}$, which yields an isomorphism $H^{0}(F^{\ast}_{+},\Ind_{U^{\infty}}^{\I^{\infty}}(C^{0}(U_{p},R)))\to C^{0}(\mathcal{G}_{p},R)$. 
The pairing
\begin{align*}
C^{0}(U_{p},R)\times\Dist(U_{p},N)\too N
\end{align*}
induces a cap product
\begin{align*}
&H^{0}(F^{\ast}_{+},\Ind_{U^{\infty}}^{\I^{\infty}}(C^{0}(U_{p},R))) \times H_{0}(F^{\ast}_{+},\cind_{U^{\prime}}^{\I^{\infty}}(\Dist(U_{p},N)))\\
\xrightarrow{\cap} &H_{0}(F^{\ast}_{+},\cind_{U^{\prime}}^{\I^{\infty}}N)\cong \bigoplus N \xrightarrow{\sum} N
\end{align*}
for every subgroup $U^{\prime}\subset U^{\infty}$ of finite index and every $R$-module N.
The direct sum decomposition  $H_{0}(F^{\ast}_{+},\cind_{U}^{\I^{\infty}}N)\cong \oplus N$ follows from Shapiro's Lemma and a strong approximation type argument.
This in turn yields a map
\begin{align}\label{partial}
\partial\colon H_{0}(F^{\ast}_{+},\cind_{U^{\prime}}^{\I^{\infty}}(\Dist(U_{p},N)))\too \Dist(\mathcal{G}_{p},N).
\end{align}

Let $\chi\colon I/F^{\ast}\to \C^{\ast}$ be an algebraic Hecke character.
It is of the form $\chi^{\prime}\left|\cdot\right|^{s}$, where $\chi^{\prime}$ is a finite order character and $s\in\Z$ and thus, its finite part takes values in a finite extension $\mathcal{E}$ of $\Q$. Let $\mathcal{R}$ be the valuation ring of $\mathcal{E}$ with respect to $\ord_{p}$.
Then the finite part of $\omega$ away from $p$ takes values in $\mathcal{R}^{\ast}$.

Let us fix an $F$-rational algebraic group $\mathbb{G}$ and an algebraic character $\dc\colon \mathbb{G}\to \mathbb{G}_{m,F}$. For every $s\in\Z$ we write $V_{s}[\dc]$ for the $\Q$-rational representation (resp.~its base change to $\mathcal{E}$) given by
$$\Res_{\F/\Q}\mathbb{G}\stackrel{\dc}{\longrightarrow}\Res_{\F/\Q}\mathbb{G}_{m,F}\xrightarrow{\mathcal{N}^{-s}} \mathbb{G}_{m,\Q},$$
where $\mathcal{N}$ is the norm character.
Let $K\subset \mathbb{G}(\A^{\infty})$ be an compact, open subgroup, which lies in the kernel of the $\mathbb{G}(F)$-invariant character
\begin{align*}
\chi\circ\dc\colon \mathbb{G}(\A)\to \C^{\ast}.
\end{align*}
We define the cohomology class
\begin{align}\label{charclass}
\begin{split}
&[\dc_{\chi}]\in \HH^{0}(\mathbb{G}(F)^{+},C(\mathbb{G}(\A^{\infty})/K,V_{s}(\dc)))\quad\mbox{via}\\
&[\dc_{\chi}](g^{\infty})=\chi(\dc(g^{\infty}))\quad \forall g^{\infty}\in \mathbb{G}(\A^{\infty}).
\end{split} 
\end{align}
Let $R$ be the completion of $\mathcal{R}$ and $E$ its field of fractions.
Assume that $K$ can be written as $K^{p}K_{p}$ with $K^{p}\subset \mathbb{G}(\A^{p,\infty})$ and $K_{p}\subset \mathbb{G}(F_{p})$.
We view $V_{s}(\dc)\otimes_{\mathcal{E}}E$ as a $\mathbb{G}(F_p)$-representation, whose underlying vector space is $E$.
We write $L_{s}(\dc)$ for the $K_{p}$-stable lattice $R\subset E=V_{s}(d)\otimes_{\mathcal{E}}E$.
It is easy to see that the image of $[\dc_{\chi}]$ under the standard isomorphism
\begin{align*}
C(\mathbb{G}(\A^{\infty})/K,V_{s}(\dc)\otimes_{\mathcal{E}}E)\xrightarrow{\cong}
C(\mathbb{G}(\A^{p,\infty})/K^{p},\Ind_{K_{p}}^{\mathbb{G}(F_p)}V_{s}(\dc)\otimes_{\mathcal{E}}E)
\end{align*}
is contained in $\HH^{0}(\mathbb{G}(F)^{+},C(\mathbb{G}(\A^{p,\infty})/K^{p},\Ind_{K_{p}}^{\mathbb{G}(F_p)}L_{s}(\dc)))$.
\end{subsection}
\begin{subsection}{Cohomological cuspidal representations}\label{cohom}
For every Archimedean place $v$ we define $K_{v}\subset G(F_{v})\cong G(\R)$ as the product of the maximal compact subgroup $\O(2n)$ and the center $Z(F_{v})=Z(\R)$ of $G(F_{v})$.
We denote by $\mathfrak{g}_{v}$ the complexification of the Lie algebra of $G(F_{v})$ and similarly we write $\mathfrak{k}_{v}$ for the complexification of the Lie algebra of $K_{v}$.
We put $K_{\infty}=\prod_{v|\infty}{K_{\infty}}$, $\mathfrak{k}_{\infty}=\oplus_{v|\infty}\mathfrak{k}_{v}$ and $\mathfrak{g}_{\infty}=\oplus_{v|\infty}\mathfrak{g}_{v}$.

Let us recall that the $(\mathfrak{g}_{\infty},K_{\infty}^{\circ})$-cohomology of a $(\mathfrak{g}_{\infty},K_{\infty}^{\circ})$-module $W$ can be computed by the Chevalley-Eilenberg complex:
\begin{align*}
\HH^{j}((\mathfrak{g}_{\infty},K_{\infty}^{\circ}),W)=\HH^{j}(\Hom_{K_{\infty}^{\circ}}(\Lambda^{\bullet}(\mathfrak{g}_{\infty}/\mathfrak{k}_{\infty}),W)).
\end{align*}
(See the book \cite{BW} of Borel and Wallach for the basics on $(\mathfrak{g}_{v},K_{v}^{\circ})$-modules and their cohomology.)
Note that there is a Künneth rule for $(\mathfrak{g}_{\infty},K_{\infty}^{\circ})$-cohomology, i.e.~ if $W=\otimes_{v} W_{v}$, where each $W_{v}$ is a $(\mathfrak{g}_{v},K_{v}^{\circ})$-module, we have
\begin{align}\label{Knth}
\HH^{j}((\mathfrak{g}_{\infty},K_{\infty}^{\circ}),W)
= \bigoplus_{\sum{j_{v}}=j}\ \ \bigotimes_{v|\infty} \HH^{j_{v}}((\mathfrak{g}_{v},K_{v}^{\circ}),W_{v}).
\end{align}
The representation $V_{v}$ is a $(\mathfrak{g}_{v},K_{v}^{\circ})$-module for all $v\in S_{\infty}$.
Given dominant weights $\mu_{v}=(\mu_{1,v},\ldots,\mu_{2n,v})\in \Z^{2n}$ for all $v \in S_{\infty}$ we let $V_{\mu_{v}}$ be the complexification of the irreducible $F_{v}$-rational representation of $G(F_{v})$ of highest weight $\mu_{v}$.
(As always, highest weight is meant with respect to the Borel group of upper triangular matrices.) 
We put $V_{\mu}=\otimes_{v\in S_{\infty}}V_{\mu_{v}}$.
This is a $\C$-rational representation of the algebraic group $\Res_{F/\Q}\mathbb{GL}_{2n}$.
\begin{Def}
The representation $V$ is cohomological of weight $\mu$ if there exists an integer $j\in \N$ such that the $(\mathfrak{g}_{\infty},K_{\infty}^{o})$-cohomology group $\HH^{j}((\mathfrak{g}_{\infty},K_{\infty}^{\circ}),V_{\infty}\otimes V_{\mu}^{\vee})$ does not vanish.
\end{Def}
\noindent From now one we assume that $V$ is cohomological of weight $\mu$ and put $V_{\al}=V_{\mu}$.
By the work of Clozel (cf.~ \cite{C} Lemma 4.9) it is known that there exists an integer $w$ - the \textit{purity weight} of $V$ - such that
\begin{align*}
\mu_{i,v}+\mu_{2n-i+1,v}=w
\end{align*}
holds for all $1\leq i\leq n$ and $v\in S_{\infty}$.
\begin{Rem} \label{cohchar}
\begin{enumerate}[(i)]\thmenumhspace
  \item If $V$ has a $(\eta,\psi)$-Shalika model, then it follows from the proof of \cite{RG} Theorem 5.3 that $\eta_{v}=\sgn^{w}|\cdot|^{w}$ for all $v\in S_{\infty}$.
	\item\label{cohheck} The central character of a cohomological representations is always an algebraic Hecke character.
\end{enumerate}
\end{Rem}
We define $q_{0}=n^2+n-1$ and set $q=dq_{0}$.
By the discussion in \cite{GRG}, Section 3.4, the cohomology group $\HH^{q_{v}}((\mathfrak{g}_{v},K_{v}^{\circ}),V_{v}\otimes V_{\mu_{v}}^{\vee})$ is $2$-dimensional if $q_{v}=q_{0}$ and vanishes if $q_{v}>q_{0}$ for every $v\in S_{\infty}$.
Hence, by the Künneth formula \eqref{Knth} we see that
\begin{align*}
\HH^{q}((\mathfrak{g}_{\infty},K_{\infty}^{\circ}),V_{\infty}\otimes V_{\al}^{\vee})
&= \bigoplus_{q_{v}=q_{0}} \bigotimes_{v|\infty} \HH^{q_{v}}((\mathfrak{g}_{v},K_{v}^{\circ}),V_{v}\otimes V_{\mu_{v}}^{\vee})\\
&\cong \bigoplus_{q_{v}=q_{0}} \bigotimes_{v|\infty} \C^{2}.
\end{align*}
Moreover, there is a natural $K_{\infty}/K_{\infty}^{\circ}$-action on $\HH^{q}((\mathfrak{g}_{\infty},K_{\infty}^{\circ}),V_{\infty}\otimes V_{\al}^{\vee})$.
The $\varepsilon$-eigenspace of this action is one-dimensional for every character $\varepsilon$ of $K_{\infty}/K_{\infty}^{\circ}$.
We fix generators $\left[V_{\infty}\right]^{\varepsilon}$ of these eigenspaces.
By Section II.3.4 of \cite{BW} we have a canonical inclusion
$$\HH^{j}(\mathfrak{g}_{\infty},K_{\infty}^{\circ},V_{\infty}\otimes V_{\al}^{\vee})\subset
\Hom_{K^{\circ}_{\infty}}(\Lambda^{j}(\mathfrak{g}_{\infty}/\mathfrak{k}_{\infty}),V \otimes V_{\al}^{\vee}).$$
Thus, after choosing a basis $(X_{i}^{\ast})$ of $(\mathfrak{g_{v}}/\mathfrak{k}_{v})^{\vee}$ and a basis $b_{v,d}^{\vee}$ of $V_{\mu_{v}}^{\vee}$ we can write $\left[V_{\infty}\right]^{\varepsilon}=\otimes_{v\in S_{\infty}}\left[V_{v}\right]^{\varepsilon_v}$ as
\begin{align*}
\left[V_{v}\right]^{\varepsilon_v}=\sum_{\underline{i}=(i_{1},\ldots, i_{q_{0}})}\sum_{l=1}^{\dim V_{\mu_{v}}^{\vee}} X_{\underline{i}}^{\ast}\otimes \varphi^{\varepsilon_v}_{v,\underline{i},l}\otimes b_{v,l}^{\vee}.
\end{align*}
Here $\varphi^{\varepsilon_v}_{v,\underline{i},l}$ are elements of $V_{v}$ and $(X_{\underline{i}}^{\ast})= X^{\ast}_{i_1}\wedge\cdots\wedge X^{\ast}_{i_{q_{0}}}$
for $\underline{i}=(i_1,\ldots,i_{q_{0}})$.
Finally, we set
\begin{align*}
\left[V_{v}\right]:=\bigoplus_{\varepsilon_v}\left[V_{v}\right]^{\varepsilon_v}
\end{align*}
and
\begin{align*}
\left[V_{\infty}\right]:=\bigoplus_{\varepsilon}\left[V_{\infty}\right]^{\varepsilon}.
\end{align*}

Our aim is to show that the distribution $\mu_{\Theta,s}$ defined in the previous chapter is a $p$-adic measure provided that $s+1/2\in \C$ is a critical point of $\pi$.
We do not want to recall the definition of criticality of a point here.
It is enough to know the following two facts: Firstly, by \cite{GRG} Proposition 6.1.1 the set of critical points of $V$ is given by
\begin{align*}
\text{Crit}(\pi)=\left\{s+1/2\in \Z+1/2 \mid -\mu_{n,v}\leq s \leq -\mu_{n+1,v}\ \forall v\in S_{\infty}\right\}.
\end{align*}
Secondly, if $s+1/2$ is critical, then for all $v\in S_{\infty}$ there is a unique $1$-dimensional $H(\C)$-stable subrepresentation $V_{s,v}$ of $V_{\mu_{v}}$ which is isomorphic to the representation given by the character $\det^{-s,s+w}$.
This is proven in \cite{GRG} Proposition 6.3.1. for the case $s=0$.
The other cases follow by twisting the representation with an integral power of the determinant.

When we defined the distribution in Section \ref{globaldist}, we had not specified the vector at infinity.
We will catch up with this now.
Let $\mathfrak{h}_{\Q}$ be the Lie-Algebra of the algebraic group $H$ over $\Q$ and $\mathfrak{k^{\prime}}_{\Q}$ the Lie subalgebra of the $\Q$-rational algebraic subgroup $H\cap Z\SO_{2n}$.
The dimension of $\mathfrak{h}_{Q}/\mathfrak{k^{\prime}}_{\Q}$ is exactly $q_{0}$.
We fix a $\Q$-basis $T_{1},\ldots,T_{q_{0}}$ of $\mathfrak{h}_{Q}/\mathfrak{k^{\prime}}_{\Q}$ and denote by $T_{i,v}$ the image of $T_{i}$ in $\mathfrak{g}_{v}/\mathfrak{k}_{v}$ for every Archimedean place $v$.
Additionally, we fix a generator $x_{s,v}$ of the subspace $V_{s,v}\subset V_{\mu_{v}}$ for every critical point $s+1/2$ of $V$ and every place $v\in S_\infty$.
Evaluation of
\begin{align*}
\left[V_{v}\right]=\sum_{\underline{i}=(i_{1},\ldots i_{q_{0}})}\sum_{d=1}^{\dim V_{\mu_{v}}^{\vee}} X_{\underline{i}}^{\ast}\otimes \varphi_{v,\underline{i},d}\otimes b_{v,d}^{\vee}
\end{align*}
at $(T_{1,v},\ldots,T_{q_{0},v},x_{s,v})$ yields an element $\varphi_{s,v}\in \pi_{v}$.
We will choose $\Phi_{s,\infty}=\otimes_{v|\infty}\varphi_{s,v}$ as the vector at infinity in the definition of the distribution $\mu_{\Theta,s}$.
It is important to know that the complex numbers
\begin{align*}
c(V_\infty,\chi_{\infty},s+1/2)=\prod_{v|\infty}\zeta(\varphi_{s,v},\chi_{v},s+1/2)
\end{align*}
do not vanish for any character $\chi_{\infty}\colon F_{\infty}^{\ast}\to \left\{\pm 1\right\}$.
Otherwise, the distribution $\mu_{\Theta,s}$ would trivially be zero.
Luckily, the non-vanishing is known by the work of Sun (cf. \cite{Sun}).
In particular, we have that for every critical point $s+1/2$ and every locally constant character $\chi_\infty$ the equality
\begin{align*}
c(V_\infty,\chi_{\infty},s+1/2)=L(\pi_{\infty}\otimes\chi_{\infty},s+1/2).
\end{align*}
holds up to a non-zero constant.
\end{subsection}
\begin{subsection}{The Eichler Shimura homomorphism}\label{ES}
In the following we are going to explain the (adelic) Eichler-Shimura map:
Let $X_v=G(F_{v})^{+}/K_{v}^{\circ}$ be the symmetric space associated to $G(F_{v})$ for $v\in S_{\infty}$.
We put $X=\prod_{v\in S_{\infty}} X_{v}$ and denote by $e$ the image of the unit element under the canonical projection $\prod_{v \in S_{\infty}}G(F_{v})^{+}\to X$.
We can naturally identify the tangent space $T_{X,e}$ of $X$ at $e$ with $\mathfrak{g}_{\infty}/\mathfrak{k}_{\infty}$.
The Eichler-Shimura map for an integer $j\geq 0$ is a $G(\A^{\infty})$-equivariant homomorphism
\begin{align*}
\HH^{q}((\mathfrak{g}_{\infty},K_{\infty}^{\circ}),V \otimes V_{\al}^{\vee})\to \HH^{0}(G(F)^{+},C^{0}(G(\A^{\infty}),\Omega^{q}_{\fdhar}(V_{\al}^{v}))),
\end{align*}
where $\Omega^{q}_{\fdhar}(V_{\al}^{\vee})$ is the space of fast decreasing harmonic $q$-differential forms on $X$ with values in $V_{\al}^{\vee}$ as defined by Borel (cf.~\cite{Bo}).
It is given as follows:
By definition $V$ is a subrepresentation of the right regular representation on $C^{\infty}(G(F)\backslash G(\A))$.
Given $\eta \in \Hom_{K_{\infty}}(\Lambda^{j}(\mathfrak{g}_{\infty}/\mathfrak{h}_{\infty}),V \otimes V_{\al}^{\vee})$ we can evaluate it on a $j$-tuple $(Y_{1},\ldots, Y_{j})$ of tangent vectors at $e$ and get 
\begin{align*}\eta(Y_{1},\ldots,Y_{j})\in C^{\infty}(G(F)\backslash G(\A), V_{\al}^{\vee}).
\end{align*}
For an element $(x,g^{\infty})\in X\times G(\A^{\infty})$ choose $g_{\infty}\in \prod_{v \in S_{\infty}}G(F_{v})$ such that $g_{\infty}e=x$. Let $Dg_{\infty}$ the differential of the action of $g_{\infty}$ on $X$.
Sending tangent vectors $Y_{1},\ldots, Y_{j}$ at a point $x\in X$ to
\begin{align*}\tilde{\eta}(g^{\infty})_{x}(Y_{1},\ldots,Y_{j})=g_{\infty}^{-1}(\eta ((Dg_{\infty})^{-1}Y_{1},\ldots,(Dg_{\infty})^{-1}Y_{j})(g_{\infty},g^{\infty}))\end{align*}
defines a differential form $\tilde{\eta}(g^{\infty})$ on $X$ with values in $V_{\al}^{\vee}$.
Since cusp forms are fast decreasing we see that we get in fact a fast decreasing differential form.
It follows from Section II.3 of \cite{BW} that every differential form in the image of the Eichler-Shimura map is closed and harmonic.

The choice of an element $$\left[V_{\infty}\right]\in \HH^{j}((\mathfrak{g}_{\infty},K_{\infty}^{\circ}),V_{\infty}\otimes V_{\al}^{\vee})$$ made at the end of Section \ref{cohom} yields a map
\begin{align*}
\ES\colon \bigotimes_{v\nmid\infty} V_{v} \longrightarrow \HH^{0}(G(F)^{+},C^{0}(G(\A^{\infty}),\Omega^{q}_{\fdhar}(V_{\al}^{v}))).
\end{align*}
\begin{Def}
\noindent Let $S$ be a finite set of finite places and $R$ a ring which contains the image of $\I^{S\cup S_{\infty}}$ under $\omega$.
We fix an algebraic subgroup $A\subset G$, which contains the center of $G$.
For every $R[A(F)^{+}]$-module $M$ and every compact, open subgroup $K\subset G(\A^{S\cup S_{\infty}})$ we define $\mathcal{C}^{S}_{\omega}(A,K,M)$ to be the $R$-module of functions $f\colon A(\A^{S\cup S_{\infty}})\to M$ such that $f(gkz)=w(z)f(g)$ for all $g\in A(\A^{S\cup S_{\infty}})$, $k\in K\cap A(\A^{S\cup S_{\infty}})$ and $z\in Z(\A^{S\cup S_{\infty}})$.
If $S$ is the empty set, we omit it from the notation.
\end{Def}
If $\Phi \in  \otimes_{v\nmid\infty} V_{v}$ is invariant under some compact, open subgroup $K\subset G(\A^{\infty})$, we see that
\begin{align*}
\ES(\Phi)\in \HH^{0}(PG(F)^{+},\mathcal{C}_{\omega}(G,K,\Omega^{q}_{\fdhar}(V_{\al}^{v})))
\end{align*}

In fact, we need a slight variant of the above construction.
Given $\Phi^{p}\in \otimes_{v\nmid p,\infty} V_{v}$ invariant under some compact, open subgroup $K^{p}\subset G(\A^{p,\infty})$ we define
\begin{align*}
\ES^{p}(\Phi^{p})\in \HH^{0}(PG(F)^{+},\mathcal{C}^{S_p}_{\omega}(G,K^{p},\Hom(V_{p},\Omega^{q}_{\fdhar}(V_{\al}^{v}))))
\end{align*}
by
\begin{align*}
\ES^{p}(\Phi^{p})(g^{p},\varphi_{p})= \ES(\Phi^{p}\otimes\varphi_{p})(g^{p},1)
\end{align*}
for $\varphi_{p}$ in $V_{p}$.
Evaluation at an element $\varphi_{p}$ of $V_{p}$, which is invariant under some compact, open subgroup $K_{p}\subset \prod_{\mathfrak{p}\in S_{p}}G(F_{\mathfrak{p}})$, induces a $PG(F)^{+}$-equivariant map
\begin{align*}
\mathcal{C}^{S_p}_{\omega}(G,K^{p},\Hom(V_{p},\Omega^{q}_{\fdhar}(V_{\al}^{v})))
\xrightarrow{\ev(\varphi_{p})}\mathcal{C}_{\omega}(G,K^{p}K_{p},\Omega^{q}_{\fdhar}(V_{\al}^{v}))
\end{align*}
such that $\ev(\varphi_{p})(\ES^{p}(\Phi^{p}))=\ES(\Phi^{p}\otimes\varphi_{p})$.

Let $\bar{X}$ the Borel-Serre bordification of $X$ with boundary $\partial X$ as constructed in \cite{BS}.
It is a smooth manifold with corners, which contains $X$ as an open submanifold.
The embedding $X\subset \bar{X}$ is a homotopy equivalence. The operation of $G(F)^{+}$ can be extended naturally to $\bar{X}$.
If $M$ is a smooth manifold with corners, we let $C^{\sing}_{\bullet}(M)$ be the complex of singular chains in $M$ and $C^{\sm}_{\bullet}(M)$ the subcomplex of smooth chains
 By Lemma 5 of \cite{Wh} continuous chains can be approximated by smooth chains.
Hence by a standard argument the inclusion $C^{\sm}_{\bullet}(M)\subset C^{\sing}_{\bullet}(M)$ is a quasi-isomorphism (see chapter 16 of \cite{Lee} for a detailed proof in the case of smooth manifolds without corners).
Using this fact for both $\bar{X}$ and its boundary $\partial X$, we see that the complex $C^{\sm}_{\bullet}(\bar{X},\partial X):=C^{\sm}_{\bullet}(\bar{X})/C^{\sm}_{\bullet}(\partial X)$ is quasi-isomorphic to the complex of relative singular chains $C^{\sing}_{\bullet}(\bar{X},\partial X)$. 
Note that these are in fact quasi-isomorphisms of complexes of $G(F)^{+}$-modules.

For every integer $j\geq 0$ there is a $PG(F)^{+}$-equivariant pairing
\begin{align*}
\Omega^{j}_{fg}(V_{\al}^{\vee})\times C^{\sm}_{j}(\bar{X},\partial X)\to V_{\al}^{\vee}
\end{align*}
given as follows: We denote by $\Delta_{j}$ the standard simplex of dimension $j$.
If $f\colon\Delta_{j}\to \bar{X}$ is a smooth chain and $\eta$ a fast decreasing differential form, we take the integral of the pullback $f^{\ast}\eta$ over the pre-image of $X$ under $f$.
If the differential form is closed, it vanishes on the image of the boundary map $C^{\sm}_{j+1}(\bar{X},\partial X)\to C^{\sm}_{j}(\bar{X},\partial X)$ by Stokes' Theorem.

Therefore we get a $G(F)^{+}$-equivariant morphism of co-complexes 
\begin{align*}
\Omega^{j}_{\fdhar}(V_{\al}^{\vee})[-j]\to \Hom(C^{\sm}_{\bullet}(\bar{X},\partial X),V_{\al}^{\vee}),
\end{align*}
which induces the following maps in (hyper-)group cohomology: 
\begin{align*}
&\HH^{0}(PG(F)^{+},\mathcal{C}_{\omega}(G,K,\Omega^{j}_{\fdhar}(V_{\al}^{v})))\\
\longrightarrow &\mathbb{H}^{j} (PG(F)^{+},\mathcal{C}_{\omega}(G,K,\Hom(C^{\sm}_{\bullet}(\bar{X},\partial X),V_{\al}^{\vee})))
\end{align*} 
and 
\begin{align*}
&\HH^{0}(PG(F)^{+},\mathcal{C}^{S_p}_{\omega}(G,K^{p},\Hom(V_{p},\Omega^{q}_{\fdhar}(V_{\al}^{v}))))\\
\longrightarrow &\mathbb{H}^{j} (PG(F)^{+},\mathcal{C}^{S_p}_{\omega}(G,K^{p},\Hom(V_{p},\Hom(C^{\sm}_{\bullet}(\bar{X},\partial X),V_{\al}^{\vee})))).
\end{align*}
We will denote the image of $\ES(\Phi)$ (resp.~$\ES^{p}(\Phi^{p})$) under the above map for $j=q$ by $\ES_{\coh}(\Phi)$ (resp.~$\ES^{p}_{\coh}(\Phi^{p})$).

\end{subsection}
\begin{subsection}{The Steinberg module}\label{steinberg}
In this chapter we recall some standard facts about the Steinberg module of an algebraic group and about Borel-Serre duality.
Let $\mathbb{G}$ be a connected split reductive group over $F$ of semi-simple $F$-rank $l\geq 1$ and $I$ the set of proper maximal $F$-rational parabolic subgroups of $\mathbb{G}(F)$.
For $\tau\in I$ let $\mathbb{P}_{\tau}$ the corresponding parabolic group.
A subset $S=\left\{\tau_{0},\ldots,\tau_{k}\right\}\subset I$ of cardinality $k+1$ is called $k$-simplex if $\mathbb{P}_{\tau_{1}}\cap\ldots\cap\mathbb{P}_{\tau_{k}}$ is a parabolic subgroup.
Let $\St_{k}$ be the free abelian group generated by the $k$-simplices on $I$.
Taking the associated simplicial complex we get a sequence of $\mathbb{G}(F)$-modules
\begin{align}\label{Steinberg}
\St_{l-1} \to  \St_{l-2}\to\cdots\to \St_{0}\to\Z\to 0.
\end{align}
\begin{Def}
The \textit{Steinberg module} $\St_{\mathbb{G}}$ of $\mathbb{G}(F)$ is the kernel of the map $St_{l-1} \to  St_{l-2}$ (where we set $\St_{-1}=\Z$ if $l=1$).
\end{Def} 
Let $\mathcal{P}_{k}$ be the set of proper $F$-rational parabolic subgroups of semi-simple $F$-rank $l-1-k$ containing a fixed Borel subgroup $\mathbb{B}(F)$ of $\mathbb{G}(F)$.
Then for $0\leq k\leq l-1$ there is a natural isomorphism of $\mathbb{G}(F)$-modules
\begin{align*}
\bigoplus_{\mathbb{P}\in\mathcal{P}_{k}}\cind_{\mathbb{P}(F)}^{\mathbb{G}(F)}\Z\stackrel{\cong}{\longrightarrow} \St_{k}.
\end{align*}
The homology of the complex (\ref{Steinberg}) can be identified with the reduced homology of the spherical building associated to $\mathbb{G}(F)$.
Since the reduced homology of the building vanishes outside the top degree (see for example \cite{BS2}) the following complex of $\mathbb{G}(F)$-modules is exact:
\begin{align*}
0\to \St_{\mathbb{G}} \to \St_{l-1} \to\cdots\to \St_{0}\to\Z\to 0
\end{align*}
\begin{Rem}
Because every parabolic subgroup of $\mathbb{G}$ contains the center $\mathbb{Z}$ of $\mathbb{G}$, we see that $\St_{\mathbb{G}}$ and $\St_{\mathbb{G}/\mathbb{Z}}$ are canonically isomorphic.
\end{Rem}
The choice of a Borel subgroup $\mathbb{B}$ with maximal torus $\mathbb{T}$ gives us the element
\begin{align*}
\tau_{\mathbb{G}}=\sum_{w\in W_{\mathbb{G}}}{w\otimes \epsilon(w)\in \St_{\mathbb{G}} \subset \Z[\mathbb{G}(F)]\otimes_{\Z[\mathbb{B}(F)]}\Z},
\end{align*}
where $W_{\mathbb{G}}$ denotes the Weyl group of $\mathbb{G}$ with respect to $\mathbb{T}$ and $\epsilon\colon W_{\mathbb{G}}\to \Z^{\ast}$ is the sign character corresponding to $\mathbb{B}$.
Now let $\mathbb{P}$ be a parabolic subgroup containing $\mathbb{B}$ and let $\mathbb{L}$ be the Levi-factor containing the torus $\mathbb{T}$.
There exists an $\mathbb{L}(F)$ equivariant map
\begin{align}\label{steinhom}
\St_{\mathbb{L}}\to \St_{\mathbb{G}}
\end{align}
which maps $\tau_{\mathbb{L}}$ to $\tau_{\mathbb{G}}$ (see Proposition 1.1 of \cite{R}).

By \cite{BS} every arithmetic subgroup $\Gamma\subset\mathbb{G}(F)$ is a virtual duality group with duality module $\St_{\mathbb{G}}$.
Let $\nu=\nu(\Gamma)$ be the virtual cohomological dimension of $\Gamma$.
It is independent of the choice of the arithmetic subgroup.
Since $\St_{\mathbb{G}}$ is $\Z$-free it follows from \cite{KB}, chapter VIII.10, that the map
\begin{align*}
\BS_{\Gamma}\colon \HH^{\bullet}(\Gamma, \Hom(\St_{\mathbb{G}},M))\stackrel{\cap e}{\longrightarrow} H_{\nu-\bullet}(\Gamma, \St_{\mathbb{G}}\otimes\Hom(\St_{\mathbb{G}},M))\stackrel{\ev}{\longrightarrow}
H_{\nu-\bullet}(\Gamma, M)
\end{align*}
is an isomorphism for every $\Gamma$-module $M$ as long as $\Gamma$ is torsion-free.
Here $e\in H_{\nu}(\Gamma, \St_{\mathbb{G}})\cong\Z$ is a fundamental class (see \cite{KB} VIII.6) and $\ev$ is the map induced by the evaluation map $\St_{\mathbb{G}}\otimes\Hom(St_{\mathbb{G}},M)\to M$.
Now let $K\subset\mathbb{G}(\A^{\infty})$ be a compact, open subgroup.
After passing to a subgroup of finite index we may assume that $\mathbb{G}(F)\cap gKg^{-1}$ is torsion-free for all $g\in\mathbb{G}(\A^{\infty})$.
For every $K$-module $M$ and every subgroup $\mathbb{G}(F)^{\prime}\subset\mathbb{G}(F)$ of finite index we get an isomorphism
\begin{align}\label{BSdual}
\BS_{\mathbb{G}}\colon \HH^{\bullet}(\mathbb{G}(F)^{\prime}, \Ind_{K}^{\mathbb{G}(\A^{\infty})}\Hom(\St_{\mathbb{G}},M))\xrightarrow{\cong}
H_{\nu-\bullet}(\mathbb{G}(F)^{\prime}, \cind_{K}^{\mathbb{G}(\A^{\infty})}M)
\end{align}
as follows:
By strong approximation the quotient $\mathbb{G}(F)^{\prime}\backslash \mathbb{G}(\A^{\infty})/K$ is finite.
Let $g_{1},\ldots,g_{r}$ be a set of representatives of this double quotient and consider the torsion-free arithmetic subgroups $\Gamma_{i}=\mathbb{G}(F)^{\prime}\cap g_{i}Kg_{i}^{-1}$.
By Shapiro's Lemma we get an isomorphism
\begin{align*}
\HH^{\bullet}(\mathbb{G}(F)^{\prime}, \Ind_{K}^{\mathbb{G}(\A^{\infty})}\Hom(\St_{\mathbb{G}},M))\xrightarrow{\cong}\bigoplus_{i=1}^{r} \HH^{\bullet}(\Gamma_{i}, \Hom(\St_{\mathbb{G}},M)).
\end{align*}
Using Borel-Serre duality for every $\Gamma_{i}$ and Shapiro's Lemma for homology afterwards yields the isomorphism (\ref{BSdual}).
We are mostly interested in the case $\mathbb{G}=H/Z$.
Theorem 11.4. of \cite{BS} gives us $\nu(\Gamma)=d(n^{2}+n-1)-2n+1=q-2n+1$ for every arithmetic subgroup $\Gamma$ of $H(F)/Z(F)$.

We can use the Steinberg module to give another description of the Eichler-Shimura map.
By Corollary 8.4.2 of \cite{BS} there is a homotopy equivalence between the boundary $\partial X$ of the Borel-Serre bordification of the symmetric space $X$ and the Bruhat-Tits building of $G(F)$ which gives a $PG(F)^{+}$-equivariant isomorphism of singular homology groups.
Since $\bar{X}$ is contractible the long exact sequence for relative homology shows that ${H}_{j+1}(\bar{X},\partial X)$ is isomorphic to the reduced homology $\tilde{H}_{j}(\partial X)$.
Thus, the complex $\Hom(C^{\sm}_{\bullet}(\bar{X},\partial X), V_{\al}^{\vee})$ is quasi-isomorphic to the complex $\Hom(\St_{G},V_{\al}^{\vee})[-2n+1]$. Therefore, we have isomorphisms 
\begin{align*}
&\mathbb{H}^{j} (PG(F)^{+},\mathcal{C}_{\omega}(G,K,\Hom(C^{\sm}_{\bullet}(\bar{X},\partial X),V_{\al}^{\vee}))) \\
\stackrel{\cong}{\longrightarrow} &\HH^{j-2n+1} (PG(F)^{+},\mathcal{C}_{\omega}(G,K,\Hom(\St_{G},V_{\al}^{\vee})))
\end{align*} 
and 
\begin{align*}
&\mathbb{H}^{j} (PG(F)^{+},\mathcal{C}^{S_p}_{\omega}(G,K^{p},\Hom(V_{p},\Hom(C^{\sm}_{\bullet}(\bar{X},\partial X),V_{\al}^{\vee})))) \\
\stackrel{\cong}{\longrightarrow} &\HH^{j-2n+1} (PG(F)^{+},\mathcal{C}^{S_p}_{\omega}(G,K^{p},\Hom(V_{p},\Hom(\St_{G},V_{\al}^{\vee})))).
\end{align*} 
We will identify $ES_{\coh}(\Phi)$ (resp.~$ES^{p}_{\coh}(\Phi^{p})$) with its image under the above isomorphism for $j=q$.
\end{subsection}
\begin{subsection}{Modular symbols}\label{groupcohom}
Let $\mathcal{E}$ be the field of definition of the finite part of $V$.
Each $V_{v}$, $v\notin S_{\infty}$, is defined over $\mathcal{E}$ and by abuse of notation we will denote its model over $\mathcal{E}$ also by $V_{v}$.
Since all embeddings $F\into \C$ factor through $\mathcal{E}$, the algebraic representation $V_{\al}$ also has a model over $\mathcal{E}$.
Again, we will denote this model by $V_{\al}$.
\begin{Def}
Let $S$ be a finite set of finite places and $R$ an algebra over the localization of the ring of integers of $\mathcal{E}$ at $\omega(\I^{S\cup S_{\infty}})$.
We fix a compact, open subgroup $K\subset G(\A^{S\cup S_{\infty}})$ and an algebraic subgroup $A\subset G$ containing the center.
Further, let $M$ be an $R[A(F)^{+}]$ module, on which the center acts via
\begin{align*}
Z(F)^{+}\longrightarrow \prod_{v\in S \cup S_{\infty}}V_{v}\xrightarrow{\omega} R^{\ast}.
\end{align*}
and $N$ an $R$-module with trivial $\underline{G}(F)^{+}$-action.
The module of $N$-valued modular symbols of weight $M$, level $K$ and character $\omega$ on $A$ is defined as
\begin{align*}
\mathcal{M}^{S}_{\omega}(A,K,M,N)
:=\HH^{q-2n+1}(PA(F)^{+},\mathcal{C}^{S}_{\omega}(A,K,\Hom_{R}(M,N)).
\end{align*}
We will omit $S$ (resp.~$\omega$) from the notation if $S=\emptyset$ (resp.~$\omega$ is trivial).
\end{Def}
\noindent Our main example of a weight module will be the $\mathcal{E}$-vector space
$$\widetilde{V}_{S}= V_{\al}\otimes\bigotimes_{v\in S}V_v.$$
For an compact, open subgroup $K\subset G(\A^{\infty})$ (resp.~$K^{p}\subset G(\A^{p,\infty})$) one can rephrase the Eichler-Shimura map (resp.~its $p$-augmented version) as a homomorphism
\begin{alignat*}{2}
&\ES_{\coh}\colon\left(\otimes_{v\notin S_{\infty}}V_{v}\right)^{K}&&\longrightarrow
\mathcal{M}_{\omega}(G,K,\St_{G} \otimes V_{\al},W)\\
\mbox{resp.~}\quad &\ES^{p}_{\coh}\colon\left(\otimes_{v\notin S_{p,\infty}}V_{v}\right)^{K^{p}}&&\longrightarrow\mathcal{M}^{S_{p}}_{\omega}(G,K^{p},\St_{G} \otimes\widetilde{V}_{S_{p}},W)
\end{alignat*}
\begin{Lem}\label{fla}
Let $S_{\infty}\subset S$ be a finite set of finite places and $K\subset G(\A^{S\cup S_{\infty}})$ an compact, open subgroup.
Then:
\begin{enumerate}[(a)]
\item The canonical map
\begin{align*}
\mathcal{M}^{S}_{\omega}(G,K,\St_{G}\otimes\widetilde{V}_{S},\mathcal{E})\otimes W\longrightarrow \mathcal{M}^{S}_{\omega}(G,K,\St_{G}\otimes\widetilde{V}_{S},W)
\end{align*}
is an isomorphism for all $\mathcal{E}$-vector spaces $W$.
\item The $\mathcal{E}^{\prime}$-module $\mathcal{M}^{S}_{\omega}(G,K,\St_{G}\otimes\widetilde{V}_{S},\mathcal{E}^{\prime})$ is finitely generated for every $\mathcal{E}$-algebra $\mathcal{E}^{\prime}$.
\end{enumerate}
\end{Lem}
\begin{proof}
More generally, we will prove the above statements for the modules $$\HH^{j}(PG(F)^{+},\mathcal{C}^{S}_{\omega}(\underline{G},K,\Hom_{\mathcal{E}}(\St_{G}\otimes \widetilde{V}_{S},W)))$$ for all $j\in \N$.\\
(a) We break the exact sequence $0\to \St_{G}\to \St_{2n-2}\to \cdots\to \St_{0}\to \Z\to 0$ into short exact sequences and consider the associated long exact sequences $\HH^{j}(\cdot,W)$ and $\HH^{j}(\cdot,E)\otimes W$.
By induction we see that it is enough to proof (a) with $\St_{G}$ replaced by $\St_{i}$, $-1\leq i \leq 2n-2$.
Since the modules $\St_{i}$ are direct sums of modules of the form $\cind_{PQ(F)}^{\PGL_{2n}(F)}\Z$ with $PQ\subset \PGL_{2n}$ (not necessarily proper) parabolic subgroups, it is enough to show that
\begin{align*}
 &\HH^{j}(PG(F)^{+},\mathcal{C}^{S}_{\omega}(G,K,\Hom_{\mathcal{E}}(\cind_{PQ(F)}^{\PGL_{2n}(F)}\Z \otimes \widetilde{V}_{S},W)))\\
 =&\HH^{j}(PG(F)^{+},\Ind_{PQ(F)}^{\GL_{2n}(F)}\mathcal{C}^{S}_{\omega}(G,K,\Hom_{\mathcal{E}}(\widetilde{V}_{S},W)))\\
=&\HH^{j}(PQ(F)^{+},\mathcal{C}^{S}_{\omega}(G,K,\Hom_{\mathcal{E}}(\widetilde{V}_{S},W)))
\end{align*}
commutes with base change
 In \cite{SS} Schneider and Stuhler construct for every $v\in S$ a finite resolution
\begin{align*}
0\longrightarrow C_{v,m}\longrightarrow\ldots\longrightarrow C_{v,0}\longrightarrow V_v \longrightarrow 0
\end{align*}
in $\mathfrak{C}_{\mathcal{E}}(G(F_{v}))$, where each $C_{v,i}$ is of the form
\begin{align*}
C_i=\cind_{K_{v,\left[i\right]}Z(F_v)}^{G(F_v)}L_i\otimes \omega_{v}
\end{align*}
with compact, open subgroups $K_{v,\left[i\right]}\subset G(F_{v})$ and $\mathcal{E}[K_{\left[i\right]}]$-modules $L_{i}$, which are finite-dimensional over $\mathcal{E}$.
A similar argument as above shows that it is enough to proof that cohomology groups of the form
$\HH^{j}(PQ(F)^{+},\mathcal{C}_{\omega}(G,K,\Hom_{\mathcal{E}}(V_{\al},W)))$
commute with base change.

By strong approximation (and the Iwasawa decomposition) the double quotient $$PQ(F)^{+}\backslash PG(\A^{\infty})/K $$ is finite.
We choose a system of representatives $g_{1},\ldots,g_{r}$ of the above double quotient and define the arithmetic subgroups
$$\Gamma_{i}=PQ(F)^{+}\cap g_{i} \left(KZ(\A^{\infty})/Z(\A^{\infty})\right)g_{i}^{-1}.$$
From Shapiro's Lemma we get the equality
\begin{align*}
\HH^{j}(PQ(F)^{+},\mathcal{C}^{S}_{\omega}(G,K,\Hom_{\mathcal{E}}(V_{\al},W)))
=&\bigoplus_{i=1}^{r} \HH^{q}(\Gamma_{i}, \Hom_{\mathcal{E}}(V_{\al},W))\\
=&\bigoplus_{i=1}^{r} \HH^{q}(\Gamma_{i}, V_{\al}^{\vee}\otimes W).
\end{align*}
Since the groups $\Gamma_{i}$ are arithmetic, they are of type (VFL).
It follows that the functor $W\mapsto H^{q}(\Gamma_{i}, V_{\al}^{\vee}\otimes W) $ commutes with direct limits (cf.~\cite{Se2}).\\
(b) can be proven in exactly the same manner.
\end{proof}
\end{subsection}

\noindent For the remainder of the article we stick to the case $S=S_{p}$.
Let $\mathcal{R}$ be the valuation ring of $\mathcal{E}$ with respect to $\ord_{p}$, $R$ its completion and write $E$ for the field of fractions of $R$.
A place $v\in S_{\infty}$ induces an embedding $F\hookrightarrow \mathcal{E}\hookrightarrow \R$ and thus a place $\p\in S_{p}$ via $\ord_{p}$.
For every $\p \in S_{p}$ we define
\begin{align*}
S_{\infty}^{\p}=\left\{v\in S_{\infty}\mid v\mbox{ induces }\p\right\}.
\end{align*}
The representation $V_{\al}$ can be written as a tensor product $V_{\al}=\otimes_{v\in S_{\infty}}V_{\al,v}$. We put
$$V_{\al}^{\p}=\otimes_{v\in S^{\p}_{\infty}}V_{\al,v}\quad
\mbox{and}\quad \widetilde{V}^{\p}_{S_{p}}=V_{\p}\otimes V_{\al}^{\p}.$$
Then $V^{\p}_{S_{p,\infty}}\otimes_{\mathcal{E}}E$ is a locally $\Q_p$-rational representation of $G(F_{\p})$.
\begin{Def}
\begin{enumerate}[(i)]\thmenumhspace
\item The representation $\widetilde{V}_{S_{p}}$ is called homologically integral if the representations $\widetilde{V}^{p}_{S_{p}}\otimes_{\mathcal{E}}{E}$ are homologically integral for all $\p \in S_{p}$.
\item A lattice $L\subset \widetilde{V}_{S_{p}}\otimes_{\mathcal{E}}E$ is called homologically of finite type if it is of the form $L=\otimes_{\p\in S_{p}}L_{\p}$, where $L_{\p}$ is a homologically of finite type lattice in $\widetilde{V}^{p}_{S_{p}}\otimes_{\mathcal{E}}E$ for each $\p \in S_{p}$.
\end{enumerate}
\end{Def}
\begin{Pro}\label{flat}
Assume that $L$ is a lattice in $\widetilde{V}_{S_{p}}\otimes_{\mathcal{E}}E$, which is homologically of finite type.
Then:
\begin{enumerate}[(a)]
\item The canonical map
\begin{align*}
\mathcal{M}^{S_{p}}_{\omega}(G,K,\St_{G}\otimes L,R)\otimes N\longrightarrow \mathcal{M}^{S_{p}}_{\omega}(G,K,\St_{G}\otimes L,N)
\end{align*}
is an isomorphism for all flat $R$-modules $N$.
\item The $R$-module ${M}^{S_{p}}_{\omega}(G,K,\St_{G}\otimes L,R)$ is finitely generated.
\end{enumerate}
\end{Pro}
\begin{proof}
Replacing the Schneider-Stuhler resolution by the resolution \eqref{res} the same proof as for Lemma \ref{fla} works.
\end{proof}
\begin{subsection}{Cohomological description of the distribution}
Besides the running assumption that $V$ is cohomological with respect to $V_{\al}$ we are going to assume in the following that
\begin{itemize}
	\item $s+1/2$ is critical for $V$,
	\item we have given a stabilization $\Theta$ of $V_{p}$ over a finite extension $\mathcal{E}^{\prime}\subset\C$ of $\mathcal{E}$ and
	\item $V$ has a $(\eta,\psi)$-Shalika model with respect to some idele class character $\eta$, whose finite part takes values in $\mathcal{E}$.
\end{itemize}
We choose a pure tensor $\Phi_{\Sigma}^{p,\infty}=\Phi^{p,\infty}_{\Sigma,m}\in \otimes_{v\notin S_p\cup S_{\infty}}V_{v}$ as in Section \ref{globaldist} and $K^{p}\subset G(\A^{p,\infty})$ an compact, open subgroup such that
\begin{itemize}
\item $K^{p}G(\mathcal{O}_{p})$ is neat, i.e.~$G(F)\cap g K^{p}G(\mathcal{O}_{p}) g^{-1}$ is torsion-free for every $g\in G(\A^{\infty})$,
\item $\Phi_{\Sigma}^{p,\infty}$ is invariant under $K^{p}$ and
\item $(\eta\circ\det^{0,1})(K_{H}^{p})=1$, where $K_{H}^{p}$ is the intersection of $K^{p}$ with $H(\A^{p,\infty})$.
\end{itemize}
Let $U^{\prime}\subset U^{\infty}$ be the subgroup generated by $U^{p}$ and the image of $K_{p}$ under the determinant. 
The main aim of this section is to construct functorial maps
\begin{align*}
\Delta_{W}^{s}\colon  \mathcal{M}_{\omega}^{S_{p,\infty}}(G,K^{p},\St_{G}\otimes\widetilde{V}_{S_{p}},W)\to  H_{0}(F^{\ast}_{+},\cind_{U^{\prime}}^{\I^{\infty}}(\Dist(U_{p}),W))
\end{align*}
for all $\mathcal{E}$-vector spaces $W$ such that $\partial (\Delta_{\C}^{s}(\ES_{\coh}^{p}(\Phi^{p})))=\mu_{\Theta,s}$ holds up to multiplication by a non-zero constant.

The homomorphism $\Delta_{W}^{s}$ is constructed in several steps:
Firstly, the map \eqref{steinhom} from $\St_{H}$ to $\St_{G}$ together with the restriction of functions yields the map
\begin{align*}
\mathcal{M}_{\omega}^{S_{p}}(G,K^{p},\St_{G}\otimes\widetilde{V}_{S_{p}},W)
\xrightarrow{\Res_{H}}
\mathcal{M}_{\omega}^{S_{p}}(H,K^{p},\St_{H}\otimes\widetilde{V}_{S_{p}},W).
\end{align*}
Secondly, $s+\frac{1}{2}$ is critical. Hence, by the discussion in Section \ref{cohom} there is unique $1$-dimensional $H(F)$-subrepresentation $V_{s}$ of $V_{\al}$, which is given by the algebraic character
$$\Res_{F/\Q}\mathbb{GL}_{2n}\xrightarrow{\det^{-s,s+w}}\Res_{F/\Q}\mathbb{G}_{m,F}\xrightarrow{\mathcal{N}}\mathbb{G}_{m,\Q}.$$
The map $\delta_{\Theta}$ as defined in \eqref{dualdos} together with the inclusion $V_{s}\hookrightarrow V_{\al}$ gives a map
\begin{align*}
\mathcal{M}_{\omega}^{S_{p}}(H,K^{p},\St_{G}\otimes\widetilde{V}_{S_{p}},W)
\xrightarrow{\delta^{\vee}_{\Theta,s}}
\mathcal{M}_{\omega}^{S_{p}}(H,K^{p},\St_{H}\otimes V_s \otimes IC_\omega^{0}(G(\mathcal{O}_{p}),\mathcal{E}),W).
\end{align*}
By Remark \ref{cohchar} we have
$$[\det^{1,-1}_{\left|\cdot\right|^{s}}]\cup[\det^{0,-1}_{\eta}] \in \HH^{0}(PH(F)^{+},\mathcal{C}_{\omega^{-1}}(H,K,V_s))$$
for the cohomology class associated to the character $\left|\det^{1,-1}\right|^{s}\eta(\det^{0,-1})$ as in \eqref{charclass}.
Thus, taking the cap product with $[\det^{1,-1}_{\left|\cdot\right|^{s}}]\cup[\det^{0,-1}_{\eta}]$ induces a map
\begin{align*}
\mathcal{M}_{\omega}^{S_{p}}(H,K^{p},\St_{H}\otimes V_s \otimes IC^{0}_{\omega}(G(\mathcal{O}_{p}),\mathcal{E}),W)&\\
\xrightarrow{[s,\eta]}
\mathcal{M}^{S_{p}}(H,K^{p},\St_{H}\otimes IC^{0}(G(\mathcal{O}_{p}),\mathcal{E}),W)&.
\end{align*}
Borel-Serre duality \eqref{BSdual} gives an isomorphism
\begin{align*}
\mathcal{M}^{S_{p}}(H,K^{p},\St_{H}\otimes IC^{0}(G(\mathcal{O}_{p}),\mathcal{E}),W)
\xrightarrow{\BS_{H}}
\HH_{0}(PH(F)^{+},\iDist(G(\mathcal{O}_{p}),W)).
\end{align*}
Finally, the pushforward map \eqref{properpush} applied to the determinant $\det\colon G(\mathcal{O}_{p})\to U_p$ induces a map
\begin{align*}
\HH_{0}(PH(F)^{+},\iDist(G(\mathcal{O}_{p}),W))
\xrightarrow {\det_{\ast}} H_{0}(F^{\ast}_{+},\cind_{U^{\prime}}^{\I^{\infty}}(\Dist(U_{p}),W).
\end{align*}
Now we can define $\Delta_{W}^{s}$ as the following composition:
\begin{align*}
\Delta_{W}^{s}=\det_{\ast} \circ \BS_{H} \circ [s,\eta] \circ \delta^{\vee}_{\Theta,s} \circ \Res_{H}.
\end{align*}
\begin{Lem}
There exists a constant $c\in \C^{\ast}$ such that 
\begin{align*}
\partial (\Delta_{\C}^{s} (\ES^{p}(\Phi_{\Sigma}^{p,\infty})_{coh}))=c\cdot \mu_{\Theta_{\Sigma},s}.
\end{align*}
\end{Lem}
\begin{proof}
We have to evaluate both sides on locally constant functions $f\colon \mathcal{G}_{p}\to\C$.
We can view such a function $f$ as an $F^{\ast}_{+}$-invariant function on $\I^{\infty}/U^{p,\infty}$.
Since $\I^{\infty}/F^{\ast}_{+}$ is compact, there exists an $m\in \N$ such that $f$ factors through the quotient $\I^{\infty}/U^{p,\infty}U_{p}^{(m)}$.
Hence, to calculate the right hand side one can replace all distribution and function spaces in the above construction by (co-)inductions of the trivial representation from open subgroups.
Now all involved cohomology groups can be written purely in terms of the (de Rham) cohomology of the associated symmetric spaces.
The claim then follows by standard computations (see for example \cite{Ha}, Section 5.3).
\end{proof}
\begin{Cor}[Rationality of the distribution]\label{distrational}
The distribution $\mu_{\Theta_{\Sigma},s}$ takes values in a finite dimensional vector space over $\mathcal{E}^{\prime}$.
For every character $\varepsilon\colon F^\ast_\infty\to \left\{\pm 1\right\}$ there exists a period $\Omega^{\varepsilon}\in \C^{\ast}$ such that 
$$\int_{\mathcal{G}_p}\chi\mu_{\Theta_{\Sigma},s}\in \mathcal{E}_{\chi}^{\prime}\Omega^{\chi_{\infty}}$$
for all characters $\chi \colon \mathcal{G}_p \to \C^{\ast}$. Here $\mathcal{E}^{\prime}_{\chi}\subset \C$ denotes the field you get by adjoining the image of $\chi$ to $\mathcal{E}^{\prime}$.
\end{Cor}
\begin{proof}
The first assertion follows directly from Lemma \ref{fla} and the lemma above.
For the second assertion let $\varphi_{\p}\in V_{\p}$ be the essential vector as defined in \cite{JPS} and $K_{\p}\subset G(\mathcal{O}_\p)$ its stabilizer.
We put $K_{p}=\prod_{\p\in S_p} K_{\p}$ and $K=K^{p}K_{p}$.
By Frobenius reciprocity we have a surjective map
$$\cind^{G(F_\p)}_{K_\p Z(F_{\p})}\omega_{\p}\stackrel{\ev}{\too} V_{\p},$$
which induces a map on modular symbols
$$\mathcal{M}_{\omega}^{S_{p}}(G,K^{p},\St_{G}\otimes\widetilde{V}_{S_{p}},W)
\stackrel{\ev^{\vee}}{\too}
\mathcal{M}_{\omega}(G,K,\St_{G}\otimes\widetilde{V}_{\al},W).
$$
By Lemma \ref{projectivity} the map
$$\delta_{\Theta}\colon IC_{\omega_\p}^{0}(G(\mathcal{O}_p),\mathcal{E}^{\prime})\mapstoo V_{\p}$$
of Section \ref{semilocal} can be lifted to a map
$$\widetilde{\delta_{\Theta}}\colon IC_{\omega_\p}^{0}(G(\mathcal{O}_p),\mathcal{E}^{\prime})\mapstoo \cind^{G(F_p)}_{K_p Z(F_{p})}\omega_{p}.$$
Thus, we get a commutative diagram of the form:
 \begin{center}
 \begin{tikzpicture}
   \path 	(0,0) 	node[name=A]{$\mathcal{M}_{\omega}^{S_{p}}(G,K^{p},\St_{G}\otimes\widetilde{V}_{S_{p}},W)$}
		(6,0) 	node[name=B]{$\mathcal{M}_{\omega}(G,K,\St_{G}\otimes\widetilde{V}_{\al},W)$}
		(0,-2) 	node[name=C]{$H_{0}(F^{\ast}_{+},\cind_{U^{\prime}}^{\I^{\infty}}(\Dist(U_{p}),W))$};
    \draw[->] (A) -- (B) node[midway, above]{$\ev$};
    \draw[->] (A) -- (C) node[midway, left]{$\Delta_{W}^{s}$}; 
		\draw[->] (B) -- (C) ;
  \end{tikzpicture} 
  \end{center}
By Section 3.9 of \cite{GRG} the pure tensor $\left(\otimes_{\p \in S_{p}} \varphi_{\p}\right) \otimes \Phi^{p,\infty}_{\Sigma}$ can be chosen to be in the rational Shalika model as defined in \textit{loc.~cit.} if $\Sigma$ contains all primes at which $V$ is ramified.
Therefore, the claim follows from multiplicity one in this case.
The general case follows from Lemma 7.1.1 of \cite{GRG}.
\end{proof}

Let $E^{\prime}$ be the completion of $\mathcal{E}^{\prime}$ with respect to $\ord_p$ and $\mathcal{R}^{\prime}$ its valuation ring.
We can regard $\Theta$ as a stabilization of $\widetilde{V}_{S_{p}}\otimes_{\mathcal{E}^{\prime}}E^{\prime}$ and thus, the notion of weak ordinarity makes sense.
\begin{Cor}[Integrality of the distribution]\label{distintegral}
\begin{enumerate}[(a)]
  \item The distribution $\mu_{\Theta_{\Sigma},s}$ is a $p$-adic measure provided that $\Theta$ is weakly ordinary with respect to $V_{\al}$ and that $V_{\p}\otimes V_{\al}$ is homologically integral.
	\item Assume that $\Theta$ is weakly ordinary with respect to $V_{\al}$ and that for every $\p$ in $S_{p}$ one of the following conditions hold:
	\begin{itemize}
	  \item $\widetilde{V}_{S_{p}}^{\p}\otimes_{\mathcal{E}^{\prime}}E^{\prime}$ is a smooth ordinary principal series representation or
		\item $F_{\p}=\Q_{p}$, $V^{\p}_{\al}$ has $p$-small weights, $V_{\p}$ is a twist of an unramified principal series representation and the central character of $\widetilde{V}_{S_{p}}^{\p}\otimes_{\mathcal{E}^{\prime}}E^{\prime}$ takes values in $\Z_{p}^{\ast}$
	\end{itemize}
	Then $\mu_{\Theta_{\Sigma},s}$ is a $p$-adic measure.
\end{enumerate}
\end{Cor}
\begin{proof}
(a) Let $L\subset \widetilde{V}_{S_{p}}$ be a lattice, which is homologically of finite type.
By Lemma \ref{integrality 2} and the discussion at the end of Section \ref{characters} we can construct maps
$$
\Delta_{N}^{s}\colon  \mathcal{M}_{\omega}^{S_{p,\infty}}(G,K^{p},\St_{G}\otimes L ,N)\to  H_{0}(F^{\ast}_{+},\cind_{U^{\prime}}^{\I^{\infty}}(\Dist(U_{p}),N))
$$
for all $R^{\prime}$-modules $N$, which agrees with the previous definition if $N$ is an $E^{\prime}$-vector space.
Therefore, the claim follows from Proposition \ref{fla}.\\
(b) follows from (a) together with Theorem \ref{latticesg} and Theorem \ref{latticeso}.
\end{proof}
In case $\mu_{\Theta_{\Sigma},s}$ is a $p$-adic measure, we can define its associated $p$-adic $L$-function as follows:
Let $\chi_{\cyc}\colon\mathcal{G}_{p}\to\Z_{p}^{\ast}$ be the cyclotomic character, i.e.~$\gamma\zeta=\zeta^{\chi_{\cyc}(\gamma)}$ holds for all $p$-power roots of unity $\zeta$ and all $\gamma\in \mathcal{G}_{p}$.
For $x\in \Z_{p}$ and $\gamma\in\mathcal{G}_{p}$ we put $\left\langle\gamma\right\rangle^{x}=\exp_{p}(x\log_{p}(\chi_{\cyc}(\gamma)))$, where $\exp_{p}$ (resp.~$\log_{p}$) is the $p$-adic exponential map (resp.~logarithm map).
The $p$-adic $L$-function attached to $\mu_{\Theta_{\Sigma},s}$ is defined by
\begin{align*}
L_{p}(\Theta_\Sigma,s,x)=\int_{\mathcal{G}_{p}}\left\langle \gamma\right\rangle^{x}\mu_{\Theta_{\Sigma},s}(d\gamma).
\end{align*}
It is an analytic function on $\Z_{p}$ with values in a finitely generated $R^{\prime}$-submodule of $E^{\prime}\otimes_{\mathcal{E}^{\prime}}\C$.

\begin{Rem}
\begin{enumerate}[(i)]\thmenumhspace
  \item Even in the non-weakly ordinary situation we can use Lemma \ref{integrality 2} (or rather its proof) to give bounds on the order of growth of our distributions in terms of \textit{slopes} of stabilizations.
	\item There should be relations between the distributions $\mu_{\Theta_{\Sigma},s}$ for different critical points $s+1/2$.
These relations together with the above mentioned bounds would enable us to construct for every stabilization of \textit{non-critical slope} a unique locally analytic distribution, which interpolates special values at all critical points.
In upcoming work of Santiago Molina and the author it is shown that these relations follow directly from Lemma \ref{integrality 2} in the $\GL_{2}$-case, thus giving a new construction of Dabrowski's $p$-adic L-function for Hilbert modular forms (cf.~\cite{Dab}).
	This enables us to generalize the work of Spie\ss~on the exceptional zero conjecture to Hilbert modular forms of higher weight.
	\end{enumerate}
\end{Rem}
\end{subsection}
\end{section}

\begin{section}{Examples}
We want to give some examples to our construction.
The natural source for these are odd symmetric powers of $p$-ordinary Hilbert modular forms over totally real fields $F$, in which $p$ is totally split.
To keep the notation simple we only deal with the case $F=\Q$.
Let $f$ be a cuspidal newform of level $\Gamma_{1}(N)$, $p\nmid N$, and weight $k\geq 2$. The associated cuspidal automorphic representation $V$ of $\GL_{2}(\A)$ is cohomological with respect to the representation $(\Sym^{k-2} \C^{2})^{\vee}$.
The local component $V_{p}$ is an unramified principal series representation of the form $\Ind^{\GL_{2}(\Q_{p})}_{B_{2}(\Q_{p})}(\chi_{1},\chi_{2})$.
Let us put $\alpha=\chi_2(p)$ and $\alpha^{\prime}=\chi_1(p)$.
Then, as explained at the end of section \ref{int} the weak ordinarity condition is equivalent to $\ord_{p}(\alpha)=1$ and $\ord_{p}(\alpha^{\prime})=k-2$.
Thus, the notion of weakly $p$-ordinarity coincides with the usual ordinarity condition at $p$.
In this case our construction gives the classical $p$-adic $L$-function as constructed for example in \cite{MTT}.

By the work of Kim and Shahidi (cf.~\cite{KS}) the symmetric cube $\Sym^{3}V$ of $V$ is known to be a cuspidal automorphic representation of $\GL_{4}(\Q)$ if $f$ is not a CM-form.
The representation $\Sym^{3}V$ is cohomological with respect to the algebraic representation of highest weight $(0,-(k-2),-2(k-2),-3(k-2))$ (see \cite{RaSh}) and the local representation at $p$ is given by $\Ind^{GL_{4}(\Q_{p})}_{B_{4}(\Q_{p})}(\chi_1^{3},\chi_{1}^{2}\chi_2^{1},\chi_1^{1}\chi_2^{2},\chi_2^{3})$.
If $f$ is $p$-ordinary, we get
\begin{align*}
\ord_{p}(\alpha^{-5}(\alpha^{\prime})^{-1}p^{k-2})=0
\end{align*}
and hence, the associated unramified stabilization is weakly ordinary.
One can combine the results of Kim (cf.~\cite{Kim}) and Jacquet-Shalika (cf.~\cite{JS}) to show that $\Sym^{3}V$ has a Shalika model (see Section 8 of\cite{GRG} for a detailed discussion).
Thus, our construction yields a $p$-adic $L$-function for every critical point of the symmetric cube of a $p$-ordinary modular form of level $\Gamma_{1}(N)$, $p\nmid N$, which is not of CM-type.

The same arguments carry over to higher odd symmetric powers as well.
Assume that $\Pi=\Sym^{2r+1}V$ is a cuspidal automorphic representation of $\GL_{2(r+1)}(\Q)$. 
Again, this implies that $\Pi$ is cohomological (cf.~\cite{RaSh}) and as in the symmetric cube case, one can show that $\Pi$ is weakly ordinary if $f$ is ordinary at $p$.
Accordingly, Banerjee and Raghuram show in \cite{BaRa} that the symmetric powers of the motive associated to $f$ are nearly $p$-ordinary, if $f$ is $p$-ordinary. 

If we would know that $\Pi$ has a Shalika model, our construction would yield a $p$-adic $L$-function for every critical point of $\Pi$.
By Proposition 8.1.4 of \cite{GRG} $\Pi$ has a Shalika model if $\Sym^{4(r-a)}V$ is an isobaric sum of cuspidal automorphic representations for all $0\leq a\leq r$.
\end{section}

\bibliographystyle{abbrv}
\bibliography{bibfile}

\end{document}